\documentclass[12pt]{amsart}
\usepackage{comment}
\usepackage{nicematrix}
\usepackage{adjustbox}
\usepackage{longtable}
\usepackage{algorithm}
\usepackage{algpseudocode}

\usepackage{pifont}
\usepackage{booktabs}
\graphicspath{{figures/}}

\usepackage{minted}

\newcommand{\Lazypic}[2]{\begin{minipage}{#1} \vspace{0.1cm} \centering {#2}\vspace{0.1cm}\end{minipage}}

\DeclareFontFamily{U}{matha}{\hyphenchar\font45}
\DeclareFontShape{U}{matha}{m}{n}{
      <5> <6> <7> <8> <9> <10> gen * matha
      <10.95> matha10 <12> <14.4> <17.28> <20.74> <24.88> matha12
      }{}
\DeclareSymbolFont{matha}{U}{matha}{m}{n}
\DeclareMathSymbol{\odiv}         {2}{matha}{"63}
\usepackage[margin=1in]{geometry}

\usepackage{amsmath,amsfonts,yhmath,hyperref}
\usepackage{amsthm}
\usepackage{amssymb}
\usepackage[latin9]{inputenc}
\usepackage{mathtools}
\usepackage{enumitem}
\usepackage{tikz-cd}
\usepackage{multirow}
\usepackage{makecell}
\usepackage{xcolor}
\usepackage{float}
\usepackage{array}
\usepackage{mathrsfs}

\usepackage{pgf,tikz}
\usetikzlibrary{calc}
\usetikzlibrary{arrows}
\usetikzlibrary[patterns]
\usetikzlibrary{shapes}
\definecolor{qqqqff}{rgb}{0.,0.,1.}
\definecolor{cqcqcq}{rgb}{0.7529411764705882,0.7529411764705882,0.7529411764705882}
\definecolor{ttqqqq}{rgb}{0.2,0.,0.}
\definecolor{qqqqff}{rgb}{0.,0.,1.}
\definecolor{xdxdff}{rgb}{0.49019607843137253,0.49019607843137253,1.}
\definecolor{zzttqq}{rgb}{0.6,0.2,0.}
\definecolor{cqcqcq}{rgb}{0.7529411764705882,0.7529411764705882,0.7529411764705882}

\definecolor{yqyqyq}{rgb}{0.5019607843137255,0.5019607843137255,0.5019607843137255}
\definecolor{uuuuuu}{rgb}{0.26666666666666666,0.26666666666666666,0.26666666666666666}
\definecolor{xdxdff}{rgb}{0.49019607843137253,0.49019607843137253,1.}
\definecolor{qqqqff}{rgb}{0.,0.,1.}

\definecolor{darkgreen}{rgb}{0.0, 0.5, 0.13}
\definecolor{darkdarkgreen}{rgb}{0.0, 0.25, 0.065}

\newcommand{\NN}{\mathbb{N}}
\newcommand{\ZZ}{\mathbb{Z}}
\newcommand{\RR}{\mathbb{R}}

\newcommand{\QQ}{\mathbb{Q}}

\newcommand{\B}{\mathcal{B}}

\renewcommand{\P}{\mathcal{P}}

\newcommand{\M}{\mathcal{M}}
\newcommand{\Z}{\mathcal{Z}}

\newcommand{\C}{\mathcal{C}}

\renewcommand{\O}{\mathcal{O}}
\newcommand{\F}{\mathcal{F}}
\newcommand{\I}{\mathcal{I}}

\newtheorem{manualtheoreminner}{Theorem}
\newenvironment{manualtheorem}[1]{%
	\renewcommand\themanualtheoreminner{#1}%
	\manualtheoreminner
}{\endmanualtheoreminner}

\DeclareMathOperator{\rk}{rk}

\DeclareMathOperator{\conv}{conv}

\newcommand{\suchthat}{\hspace{4pt}|\hspace{4pt}}

\usepackage{youngtab}

\newtheorem{theo}{Theorem}[section]
\newtheorem{theorem}[theo]{Theorem}

\newtheorem{proposition}[theo]{Proposition}

\newtheorem{corollary}[theo]{Corollary}

\newtheorem{lemma}[theo]{Lemma}

\newtheorem{proposition-definition}[theo]{Proposition-Definition}

\theoremstyle{definition}

\newtheorem{definition}[theo]{Definition}

\theoremstyle{remark}
\newtheorem{remark}[theo]{Remark}
\newenvironment{rem}[1]{
    \begin{remark}#1}{
    \xqed{\blacklozenge}\end{remark}
}

\theoremstyle{remark}
\newtheorem{example}[theo]{Example}

\newcommand{\xqed}[1]{
    \leavevmode\unskip\penalty9999 \hbox{}\nobreak\hfill
    \quad\hbox{\ensuremath{#1}}}

\newcommand{\chainZ}{\mathscr{C}_{\mathcal{Z}}}
\usepackage{ytableau}

\usepackage{cleveref}

\usepackage{bm}
\usepackage{bbm}
\usepackage[scr=boondoxo]{mathalpha}

\title{The polytope of all matroids in ranks 2 and 3}
\author{Narayan Collins and Victoria Schleis}
\date{May 2026}

\begin{document}
\begin{abstract}
     We give explicit recursive constructions for the \emph{polytope of all matroids} $\Omega_{r,n}$ in ranks 2 and 3 for all ground set sizes. This polytope was introduced in recent work by Ferroni and Fink as a tool for checking positivity conjectures for valuative invariants. We supplement our theoretical construction by an implementation, which allows for the computation of $\Omega_{2,n}$ for $n\leq 33$ and $\Omega_{3,n}$ for $n\leq 10$. Further, we compute Schubert expansions for all isomorphism classes of matroids of rank $2$ up to $n = 80$, and for rank $3$ up to $n = 11$. 
\end{abstract}
\maketitle
\section{Introduction}
In matroid theory, and, in particular, in combinatorial Hodge theory, many open problems can be reinterpreted as the non-negativity of certain matroid invariants. 

To this end, Ferroni and Fink \cite{FerroniFink2025} introduced the \emph{polytope of all matroids} as the parameter polytope for Schubert expansions of isomorphism classes of matroids of a given rank and on a given (finite) ground set.  The decomposition of the indicator functions of matroid polytopes as integer expressions of indicator functions of Schubert matroids goes back to Derksen and Fink \cite{DerksenFink2010}. 

On the polytope of all matroids, positivity conjectures of valuative invariants can be tested efficiently -- if they are satisfied on all matroid isomorphism classes whose associated lattice points as Schubert expansions are vertices of the polytope of all matroids (called \emph{extremal matroids}), they are satisfied on all matroids.

In this paper we explicitly compute the polytope of all matroids in ranks 2 and 3. For matroids of rank 2 provide a careful analysis of all Schubert expansions for all possible matroid isomorphism classes in Section \ref{sec poam rk 2}. As a result, we give a recursive construction of the polytope of all matroids for arbitrary finite ground sets $[n]$.

\begin{manualtheorem}{A}\label{thm schubert polytope intro}[{c.f. Theorem \ref{thm schubert polytope}}]
The polytope $\Omega_{2,n} \coloneqq \conv(\O_{2, n})$, defined as the convex hull of the columns of the matrix $\O_{2,n}$ as described in \eqref{eq o2n} is the polytope of all matroids of rank $2$ on $n$ elements.
\end{manualtheorem}
Through our recursive construction of $\O_{2,n}$, we can gain more insight into extremal matroids of rank two. In general, Ferroni and Fink show that there are more isomorphism classes of extremal matroids than there are vertices of $\Omega_{r,n}$. However, we show that in rank 2, this is a one-to-one correspondence in Corollary \ref{cor extremal}, and provide a database of these extremal matroids in our computational supplement. 

In rank $3$, the polytope of all matroids is more challenging to compute. In Section \ref{sec poam rk 3} we analyse Schubert expansions of matroids of rank 3. Here, our recursive algorithm works by first constructing all simple matroids of rank 3 and computing their Schubert coefficients. Then, we generalise by tracking the changes to Schubert coefficients arising when constructing the simplification of a matroid. Incrementally, this allows us to compute all possible Schubert expansions of rank 3 matroids, and hence the entire polytope of all matroids. 

\begin{manualtheorem}{B} \label{poam 3 intro}[{c.f. Theorem \ref{thm main rk 3}}]
    Let $n \geq 4$, and let $\O_{3,n}$ be the matrix computed in Algorithm \ref{alg:cap}. The polytope $\Omega_{3,n} \coloneqq \conv ( \O_{3,n})$ is the polytope of all matroids of rank $3$ on $n$ elements.
\end{manualtheorem}

We supplement our theoretical results with implementations of the recursive formulas in \texttt{julia}, using \texttt{Oscar} \cite{Oscar} for the computation of the polytope from $\O_{k,n}$. Our implementation can be found at 

\begin{center}
\url{https://github.com/VictoriaSchleis/polytope_of_all_matroids}. 
\end{center}
Our constructions allow us to compute the polytope of all matroids for surprisingly large ground sets: 
\begin{itemize}
    \item In rank $2$ the polytope itself can be computed up to ground set size $33$, and its generating lattice points can be computed up to ground set size 80, see Section \ref{sec implementation}. Combined with Corollary \ref{cor extremal} this allows us to substantially expand the data on extremal matroids provided by Ferroni and Fink.  
    \item In rank $3$, the polytope itself can be computed up to ground set size $10$, whereas its generating lattice points can be additionally computed for $n=11$, see Section \ref{subsec poam rk 3}.
\end{itemize}

\begin{center}
    \textsc{Acknowledgements}
\end{center}
We thank Alex Fink and Yue Ren for helpful conversations, and thank Luis Ferroni for asking the questions that inspired this paper. 

\emph{Funding. } V.S. was funded by the UKRI FLF grant \emph{Computational Tropical Geometry}  \\ MR/Y003888/1. During parts of this project, V.S. was a member of the Institute for Advanced Study, Princeton,  funded by the Charles Simonyi Endowment. 

\section{Schubert expansions}
After a brief review of notation, Schubert matroids,  and cyclic chain lattices in Section \ref{sec cyclic chain lattices},  we discuss Schubert expansions in Section \ref{sec Schubert expansion}. These will be the main ingredient in our construction of the polytopes of all matroids. We conclude by recalling the definition of the polytope of all matroids. We assume the reader to be familiar with basic matroid theory and refer  to \cite{oxley2011} for a thorough introduction. 

Throughout, we write $[n]$ for the set $\{1, 2, \dots, n\} \subset \mathbb{N}$, and $[a, b]$ for the set $\{a, a+1, \dots, b-1, b \} \subset \mathbb{Z}$. As is standard, when it is clear from context, we may omit braces and commas when describing subsets of $[n]$.

Let $M$ be a matroid of rank $r$ over $[n]$. We then write $\rk_M$ for its rank function and
 $\mathcal{I}(M)$ for its set of \emph{independent sets}, i.e. $$\I(M) = \{I \subseteq [n] \suchthat  \rk_{M}(I) = |I|\}.$$
  \emph{Circuits} are minimal dependent sets of matroids, i.e., $$\C(M) = \{C \subseteq[n]\suchthat \{ C\setminus i \} \in \I(M) \text{ for all } i\in C \}.$$ 
  Further, the \emph{lattice of flats} $\mathcal{L}(M)$ is the lattice (ordered by inclusion) of sets closed under the rank function, i.e., $F\in \mathcal{L}(M)$ if and only if $\rk(F)< \rk(F\cup i)$ for all $i\in [n]\setminus F.$ 

We consider two special subsets of the lattice of flats. We write  $$L(M) = \{i \suchthat \rk_M(i) = 0\}$$ for its set of \emph{loops} and
 $P(M)$ for the set of \emph{parallel classes} of $M$, i.e., 
    $$P(M) = \{F \setminus L(M) \suchthat F \text{ is a flat of rank } 1\}.$$

The \emph{base polytope} $\P(M)$ is the convex hull of the indicator vectors of elements in $\B(M)$, i.e. 
$
\mathcal{P}(M) = \conv_{B \in \mathcal{B}}\left( \sum_{i \in B} \bm{e}_{i} \right) \subset \mathbb{R}^{n},$
where $\bm{e}_{i}$ is the $i$-th standard basis vector.

\subsection{Cyclic chain lattices} \label{sec cyclic chain lattices}
In this subsection, we discuss cyclic flats, their lattices, and their chain lattices. 

\begin{definition}
    A flat of a matroid is called \emph{cyclic} if it is a union of circuits. 
\end{definition}

Similar to the flats of a matroid, the cyclic flats can be arranged as a lattice. Denoting the \emph{lattice of cyclic flats} of a matroid $M$ as $\mathcal{Z}(M)$.
\emph{Schubert matroids} are matroids whose lattice of cyclic flats form a chain.

\begin{remark}\label{schubertcount}
   Schubert matroids of size $n$ and rank $k$ are in bijection to increasing lattice paths in a rectangle  from $(0, 0)$ to $(n-k, k)$.  We write $S(x_{1}, \dots, x_{k})$ for the Schubert matroid of rank $k$ in bijection to the lattice path with corners at $(x_1-1, 0), (x_2-2, 1), \dots, (x_k-k, k)$.

   The bijection further implies that $|\mathscr{S}_{n, k}| =\binom{n}{k}$. Moreover, the number of loopless Schubert matroids on $n$ of rank $k$ is $\binom{n-1}{k-1}$, corresponding to the lattice paths whose first step is in vertical direction.   
\end{remark}

\begin{theorem}[{\cite[Theorem 3.2]{BoninDeMier2008}}] \label{thm properties z(m)}
Let $\mathcal{Z} \subseteq \mathcal{P}^{[n]}$, and let $r : \mathcal{Z} \rightarrow \mathbb{Z}$ be an integer valued function. Then there is a matroid $M$ of rank $r$ on $[n]$ with $ \mathcal{Z}(M) = \mathcal{Z}$  if and only if \label{caxioms}
\begin{itemize}
\item[\textbf{(Z0)}] $\mathcal{Z}$ is a lattice under inclusion.
\item[\textbf{(Z1)}] $r(0_{\mathcal{Z}})=0$.
\item[\textbf{(Z2)}] $0 < r(Y) - r(X) < |Y \setminus X|$ for all $X, Y \in \mathcal{Z}$ with $X \subset Y$.
\item[\textbf{(Z3)}] For all $X, Y \in \mathcal{Z}$,
\begin{align*}
r(X) + r(Y) \geq r(X \vee Y) + r(X \wedge Y) + |(X \cap Y) \setminus (X \wedge Y)|
\end{align*}
\end{itemize}
\end{theorem}

\begin{lemma}\label{crank}
Let $M$ be a matroid of rank $k$ over $[n]$. A set $S \subseteq [n]$ that contains all circuits of $M$ is a flat. Let $C$ denote the union of all circuits. Then, $C$ is the maximal element of $\mathcal{Z}(M)$, and $\rk(M) = r(C) +n - |C|$. 
\end{lemma}
\begin{proof}
Assume $S$ is not a flat. Then, there exists $e\in [n]\setminus S$ such that $\rk(S\cup e) = \rk(S)$. By assumption, $e$ is not contained in any circuit, and thus a coloop.  Since $e$ is a coloop, $\rk(S\cup e) > \rk(S)$, contradiction. All elements of the ground set not contained in $C$ are coloops. Thus, the rank of $M$ can be expressed as $\rk(M) = r(C) +n - |C|$.\qedhere
\end{proof}

\begin{definition}
Let $P$ be a finite poset containing a minimal and maximal element. The \emph{chain lattice} $\mathscr{C}(P)$ is the atomic lattice with maximal element $\hat{\bm{1}}$ whose elements are the chains of $P$ that include the minimal and maximal element, ordered by inclusion.

Given a lattice of cyclic flats $\mathcal{Z}(M)$ of a matroid $M$, we denote the chain lattice of $\mathcal{Z}(M)$ with $\mathscr{C}_{\mathcal{Z}}(M)$ and call it the \emph{cyclic chain lattice} of $M$.
\end{definition}
\begin{example}\label{ex1}
We consider $M = U_{1, 2} \oplus U_{1, 2}$. Since $M$ has rank 2, all circuits of $M$ contain at most 3 elements. The only circuits of the matroid are $\{1, 2\}$ and $\{3, 4\}$, so $\mathcal{Z}(M)$ is shown on the left of Figure \ref{czm}. To construct the cyclic chain lattice $\mathscr{C}_{\mathcal{Z}}(M)$, we extract the chains of $\mathcal{Z}(M)$ containing $\emptyset$ and $[4]$. Since $M$ has no loops, the minimal element of $\mathscr{C}_{\mathcal{Z}}(M)$ is $\emptyset  \prec [4]$. In this case, there are only 2 other chains, so the chain lattice is easily constructed, seen on the right of Figure \ref{czm}.
\end{example}

\begin{figure}
\centering
\begin{tabular}{c c}
\Lazypic{4cm}{
\begin{tikzpicture}[scale = 0.75]
\draw (0.25, -0.75) -- (0.75, -0.25);
\draw (-0.25, -0.75) -- (-0.75, -0.25);
\draw (0.75, 0.35) -- (0.35, 0.75);
\draw (-0.75, 0.35) -- (-0.35, 0.75);
\draw (0, -1) node{$\emptyset$};
\draw (-1, 0) node{$1 \  2$};
\draw (1, 0) node{$3 \ 4$};
\draw(0, 1) node{$[4]$};
\end{tikzpicture}}&\Lazypic{8cm}{
\begin{tikzpicture}[scale = 0.75]
\draw (0, 3) node{$\hat{\bm{1}}$};
\draw (-3,1.5) node{$\begin{pmatrix} [4] \succ  1 \ 2 \succ  \emptyset \end{pmatrix}$};
\draw (3, 1.5) node{$\begin{pmatrix} [4] \succ 3 \ 4 \succ  \emptyset \end{pmatrix}$};
\draw (0,0) node{$\begin{pmatrix} [4] \succ  \emptyset \end{pmatrix}$};
\draw (-2.25, 2) -- (-0.25, 2.75);
\draw (2.25, 2) -- (0.25, 2.75);
\draw(-2.25, 1) -- (-0.5, 0.5);
\draw(2.25, 1) -- (0.5, 0.5);
\end{tikzpicture}}
\end{tabular}
\caption{To the left, the lattice of cyclic flats $\mathcal{Z}(M)$, and to the right, the  cyclic chain lattice $\mathscr{C}_{\mathcal{Z}}(M)$, both seen in Example \ref{ex1}.} \label{czm}
\end{figure}
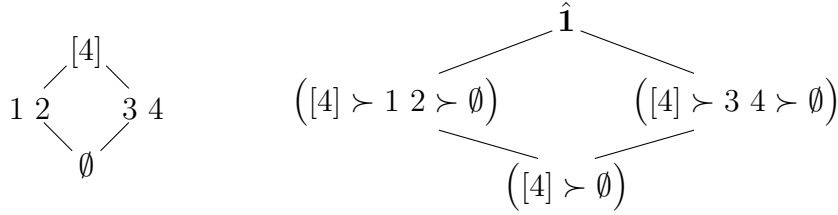

\begin{rem}
For a matroid $M$, every chain of $\mathcal{Z}(M)$ containing the maximal element $M$ and the minimal element (the union of all loops in the matroid) corresponds to a Schubert matroid of the same size and rank. 
\end{rem}

\subsection{Schubert expansion and the polytope of all all matroids} \label{sec Schubert expansion}
We now recall Schubert expansions and the definition of the polytope of all matroids. We begin by recalling the definition of the M\"obius function of a lattice. 

\begin{definition}
Let $\mathcal{L}$ be a lattice. Then for any $a, b \in \mathcal{L}$, we define the \emph{M\"obius function} 
\begin{align*}
\mu(a, b) = \begin{cases}
1 & \text{if} \ a=b ;\\
- \sum_{a \leq k < b} \mu(a, k) & \text{if} \ a<b; \\
0 & \text{otherwise}.
\end{cases}
\end{align*}

\end{definition}

\begin{definition}\label{def iso classes matroid}
    Let $M$ and $N$ be matroids of rank $r$ on $[n]$. We say that $M$ and $N$ are \emph{isomorphic} if there exists a bijection $f:[n]\rightarrow[n]$ such that $f(B) \in \mathcal{B}(N)$ for all $B\in \mathcal{B}(M)$ and $f(B') \in \mathcal{B}(M)$ for all $B'\in \mathcal{B}(N)$. 
    
    We denote the set of all isomorphism classes of matroids of rank $r$ on $[n]$ by $\mathscr{M}_{r,n}$, and the set of all isomorphism classes of Schubert matroids of rank $r$ over $[n]$ by $\mathscr{S}_{r,n}$.
\end{definition}

Let $X\subseteq \RR^n$ be a set and let $\overline{\mathbbm{1}}(X): \RR^n \rightarrow \ZZ$ denote the standard indicator function of $X$. The \emph{symmetrized indicator function} $\mathbbm{1}(X)\RR^n \rightarrow \QQ$ is defined by 
$$\mathbbm{1}(X)(x) = \frac{1}{n!}\sum_{\sigma\in S_n}\overline{\mathbbm{1}}(X)(\sigma(x)),$$
where $S_n$ denotes the $n$th symmetric group. Using this notion, we have that $\mathbbm{1}(\mathcal{P}(M)) = \mathbbm{1}(\mathcal{P}(\sigma(M)))$. This implies that $\mathbbm{1}(\mathcal{P}(M))$ is well-defined on isomorphism classes $[M]\in \mathscr{M}_{r,n}.$

Using the work of Ferroni and Fink, the symmetrized indicator function of the base polytope for any matroid can be decomposed as a linear combination of symmetrized indicator functions of polytopes of Schubert matroids:
\begin{theorem}[{\cite[Corollary 2.17]{FerroniFink2025}}]\label{thm: schubert_moebius}
Let $M \in \mathscr{M}_{r, n}$ be an isomorphism class of matroids. For each element $\C \in \mathscr{C}_{\mathcal{Z}}(M) \setminus \hat{\bm{1}}$, define $\lambda_{\C} := -\mu(\C, \hat{\bm{1}})$, and let $S_{\C} \in \mathscr{S}_{r, n}$ be the Schubert matroid isomorphism class whose lattice of cyclic flats is $\C$. Then
\begin{align*}
\mathbbm{1}(\mathcal{P}(M)) = \sum_{\C \in \mathscr{C}_{\mathcal{Z}}(M) \setminus \hat{\bm{1}}} \lambda_{\C}\mathbbm{1}\big(\mathcal{P}(S_{\C})\big),
\end{align*}
and this expression is unique across all Schubert matroids in $\mathscr{S}_{r, n}$.
\end{theorem}
This decomposition allows the construction of the polytope of all matroids. For a Schubert matroid $S$, we write
\begin{align*}
\lambda_{S} = \sum_{\substack{\C \in \mathscr{C}_{\mathcal{Z}}(M) \\ \C \cong \mathcal{Z}(S)}} \mu(\C, \hat{\bm{1}}).
\end{align*}
Using this decomposition, the sum in Theorem \ref{thm: schubert_moebius} can instead be rewritten in terms of Schubert matroids. We conclude by defining the polytope of all matroids.
\begin{definition}
For given $r, n \in \ZZ_{\ge 0}$,  label the elements of $\mathscr{S}_{r,n}$ by $S_{i}$, with $|\mathscr{S}_{r, n}| = N$. Now, for each $M \in \mathscr{M}_{r, n}$, we define $\bm{o}_{M} \in \mathbb{R}^{N}$ to be
\begin{align*}
(\bm{o}_{M})_{i} = a_{i}, \ \text{where} \ \mathbbm{1}(\mathcal{P}(M)) = \sum_{i=1}^{N} a_{i}\mathbbm(\mathcal{P}(S_{i})).
\end{align*} 
Then, the polytope of all matroids $\Omega_{r, n}$ is the convex hull
\begin{align*}
\Omega_{r, n} = \conv_{M \in \mathscr{M}_{r, n}}(\bm{o}_{M}).
\end{align*}
 For a set $\mathscr{A} \subseteq \mathscr{M}_{r, n}$, we write $\O(\mathscr{A})$ for the matrix of $(\bm{o}_{i})_{i \in \mathscr{A}}$. 
\end{definition}

\section{The polytope of all matroids of rank 2 } \label{sec poam rk 2}
In this section, we construct the polytope of all matroids of rank $2$. We start by classifying Schubert matroids of rank 2 in Theorem \ref{thm schubert rk 2} and continue by constructing non-Schubert matroids in Lemma \ref{prop non schubert rk 2}. We then characterise their Schubert expansions in Theorem \ref{thm schubert decomp rk 2} and conclude by constructing the polytope in Theorem \ref{thm schubert polytope}. We finish this section by explicitly determining the polytope of all matroids computationally for high ground set sizes in Section \ref{sec implementation}.

\subsection{Lattices of cyclic flats in rank 2}
To construct the polytope of all matroids, we describe the lattices of cyclic flats for all matroids of rank 2. We start by constructing the lattices of Schubert matroids. 

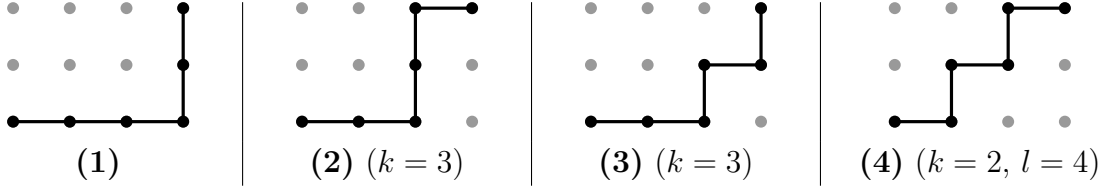
\begin{figure}
    \centering\begin{tabular}{c@{\hskip 0.7cm} | @{\hskip 0.7cm}c @{\hskip 0.7cm}|@{\hskip 0.7cm} c @{\hskip 0.7cm}|@{\hskip 0.5cm} c}
        \begin{tikzpicture}[scale = 0.75]
           \foreach \x in {0,1,2,3}
    \foreach \y in {0,1,2}
     {
    \fill[opacity = 0.4] (\x,\y) circle (3pt);
    }
    \draw[very thick] (0,0) -- (3,0) -- (3,2);
    \fill (0,0) circle (3pt);
    \fill (1,0) circle (3pt);
    \fill (2,0) circle (3pt);
    \fill (3,0) circle (3pt)
    ;\fill(3,1) circle (3pt);
    \fill(3,2) circle (3pt);
    
        \end{tikzpicture} & 
        \begin{tikzpicture}[scale = 0.75]
           \foreach \x in {0,1,2,3}
    \foreach \y in {0,1,2}
     {
    \fill[opacity = 0.4] (\x,\y) circle (3pt);
    }
    \draw[very thick] (0,0) -- (2,0) -- (2,2) -- (3,2);
    \fill (0,0) circle (3pt);
    \fill (1,0) circle (3pt);
    \fill (2,0) circle (3pt);
    \fill (2,2) circle (3pt)
    ;\fill(2,1) circle (3pt);
    \fill(3,2) circle (3pt);
    
        \end{tikzpicture}& 
        \begin{tikzpicture}[scale = 0.75]
           \foreach \x in {0,1,2,3}
    \foreach \y in {0,1,2}
     {
    \fill[opacity = 0.4] (\x,\y) circle (3pt);
    }
    \draw[very thick] (0,0) -- (2,0) -- (2,1) -- (3,1) --(3,2);
    \fill (0,0) circle (3pt);
    \fill (1,0) circle (3pt);
    \fill (2,0) circle (3pt);
    \fill (3,1) circle (3pt)
    ;\fill(2,1) circle (3pt);
    \fill(3,2) circle (3pt);
        \end{tikzpicture}& 
        \begin{tikzpicture}[scale = 0.75]
           \foreach \x in {0,1,2,3}
    \foreach \y in {0,1,2}
     {
    \fill[opacity = 0.4] (\x,\y) circle (3pt);
    }
    \draw[very thick] (0,0) -- (1,0) -- (1,1) -- (2,1) -- (2,2) -- (3,2);
    \fill (0,0) circle (3pt);
    \fill (1,0) circle (3pt);
    \fill (2,2) circle (3pt);
    \fill (1,1) circle (3pt)
    ;\fill(2,1) circle (3pt);
    \fill(3,2) circle (3pt);
        \end{tikzpicture} \\
        
        \textbf{(1)} & \textbf{(2)} ($k = 3$)  & \textbf{(3)}   ($k = 3$) & \textbf{(4)} ($k=2$, $l = 4$) 
    \end{tabular}
    \caption{Examples of the four types of lattice paths giving rise to the four different types of Schubert matroids in Theorem \ref{thm schubert rk 2}.}
    \label{fig lattice paths}
\end{figure}

\begin{theorem}\label{thm schubert rk 2}
    In rank 2, there are $ \frac{n}{2}(n-1)$ Schubert matroids, which can be separated into four classes using the combinatorics of their lattice paths.
\end{theorem}
\begin{proof}
The first part of the theorem follows directly from Remark \ref{schubertcount}. To characterize the different combinatorial types of lattices, we consider different types of lattice paths, depicted in Figure \ref{fig lattice paths}. 

\begin{enumerate}
\item The matroid $S(n-1, n)$ has only one basis, $\{n-1, n\}$. Its only circuits are the set of loops, so  $L(S(n-1, n)) =  [n-2]$ is the only element of $\mathcal{Z}(S(n-1, n))$.

\item The matroid $S(k, k+1 )$ has loops $L(S(k,k+1)) = [k-1]$. The elements of $[k, n]$ are pairwise independent, so no parallel classes exist, and any 3 element subset of these sets is a 3-circuit. Thus, the entire ground set is a union of circuits, and so $\mathcal{Z}(S(k, k+1 ))$ has no flats of rank 1.  
\end{enumerate}

Now, we consider lattice paths without vertical edges of length 2. 
\begin{enumerate}
\item[(3)] For $k<n-1$, we have $L(S(k,n)) = [k-1]$, as in type (2), but we have $$\B(S(k,n)) = \{ \{ a, n\} \suchthat a \in [k, n-1] \}.$$ Then, $[k, n-1]$ forms a parallel class. Further, $n$ is not included in any circuit, as it is the only coloop. 
Thus,  $\mathcal{Z}(S(k,n))$ is the lattice $[k-1] \prec [n-1]$.

\item[(4)] If $1 \leq k_1 < k_2 < n-1$, then $L(S(k_1,k_2)) = [k_1-1]$ as before. Further, each pair of elements in $[k_1, k_2-1]$  is dependent, i.e., $[k_1, k_2-1]$ is a parallel class. The elements of $[k_2, n]$ are the coloops of $S(k_1,k_2)$. This means that a set formed from any element of $[k_1, k_2-1]$ along with 2 elements in $[k_2, n]$ forms a circuit. So $[n]$ itself is a union of circuits, and $\mathcal{Z}(S(k_1,k_2))$ is the lattice $[k_1-1]\prec [k_2-1]\prec[n]$.\qedhere
\end{enumerate}

\end{proof}

To construct the polytope of all matroids, we now construct the lattices of cyclic flats for all non-Schubert matroids. 


Hence,  on a finite set $[n]$ we can construct a non-Schubert matroid as follows:
\begin{enumerate}
\item Consider two disjoint subsets $ L \sqcup C \subseteq [n]$  where  $|C| \geq 4$. 
Here, $L$ is the set of loops and $C$  the set of loopless rank $1$ subsets.
\item Separate $C$ into $n\geq 2$ disjoint sets $\{C_i\}_{1 \leq i \leq n}$, requiring that $|C_i| \geq 2$.
\item Construct the lattice $\mathcal{L}$ (ordered by inclusion) on the sets $\{ L, L \sqcup C_i, [n]\}$.
\end{enumerate}

\begin{lemma}\label{prop non schubert rk 2}
    The lattice $\mathcal{L}$ constructed above is the lattice of cyclic flats $\mathcal{Z}(M)$ of a non-Schubert matroid $M$ of rank $2$ on $[n]$. 
\end{lemma}
\begin{proof}
We verify the axioms of Theorem \ref{caxioms}:
\begin{itemize}
\item[(Z0)] By construction, $\mathcal{L}$ is a poset with respect to inclusion. Further,  there is both a minimal element $L$ and a maximal element $[n]$, which are of rank 0 and 2 respectively. Thus, it remains to be shown that any two elements of rank 1 have a join and a meet. 

Let $X=L \sqcup C_m$ and $Y = L \sqcup C_n$, for $m \neq n$. To construct the join, note that for any set $Z\in \mathcal{L}$ containing both $X$ and $Y$, we have $Z \supseteq X \cup Y = L \sqcup C_n \sqcup C_m$. The only element of $\mathcal{L}$ that contains both $C_n$ and $C_m$ is $[n]$. Hence $X \vee Y = [n]$. 

For the meet, note that $X \wedge Y =  X \cap Y = L$. Hence $\mathcal{L}$ is a lattice under inclusion.
\item[(Z1)] $\hat{\bm{0}} = L$ since it is contained in every element of $\mathcal{L}$, and $r(L) = 0$ by definition.
\item[(Z2)] Since the lattice is ordered by inclusion, $X \subset Y \implies r(X) < r(Y)$.

If $r(Y) = 2$ and $r(X) = 0$ we have $Y=[n], X=L$, and since $[n] \supseteq C$ and $|C| \geq 4$ 
\begin{align*}
r([n]) - r(L) = 2  < 4 \leq |[n] \setminus L||.
\end{align*}
Otherwise, we have that $r(Y)-r(X) = 1$. Then, one of the sets $X$ and $Y$ is of the form $L \sqcup C_n$ and the other set is either $[n]$ or $L$. Since both $2 \leq |[n] \setminus (L \sqcup C_n)|$ and $2 \leq |C_n| = | (L \sqcup C_n) \setminus L|$,  the axiom follows.

\item[(Z3)]  In (Z1), we saw that for all $X, Y \in \mathcal{L}$, $X \wedge Y = X \cap Y$. As a result, we need only prove the inequality
\begin{align*}
r(X) + r(Y) \geq r(X \vee Y) + r(X \wedge Y).
\end{align*}
If $X \subseteq Y$, then $X \vee Y = Y$ and $ X \wedge Y = X$. Thus, the inequality is satisfied.  If $X$ and $Y$ are incomparable, then $X = L \sqcup C_n$ and $Y = L \sqcup C_m$ for some $m$ and $n$. Then $Y \vee X = [n] $ and $ Y \wedge X = L$, and $r(X)=r(Y)=1$. Hence
$r(X)+r(Y) = 2 = r(L) + r([n])$
as required. \qedhere
\end{itemize}

\color{black}
\end{proof}

\begin{proposition}\label{prop non schubert lof rk 2}
     Let $M$ be a coloopless non-Schubert matroid $M$ of rank $2$ on $[n]$. Its lattice of cyclic flats $\mathcal{Z}(M)$ has the following properties:
\begin{enumerate}
\item The entire ground set $[n]$ is the maximal cyclic flat.
\item Each cyclic flat $P$ of rank 1 satisfies $|P|\geq |L(M)|+2$.
\end{enumerate}
    In particular, the lattice of cyclic flats of every coloopless non-Schubert matroid of rank 2 arises by the construction described above. 
\end{proposition}
\begin{proof}
To see (1), we can equivalently show that the rank of the maximal element $\hat{\bm{1}}$ of must be 2, since Corollary \ref{crank} then imposes that $|[n] \setminus \hat{\bm{1}}| = 0$.

Assume $r(\hat{\bm{1}}) = 0$ or  $r(\hat{\bm{1}}) = 1$. Then, the lattice forms a chain using Theorem \ref{thm properties z(m)} as $\mathcal{Z}(M)$ has a unique minimal and maximal element.  
In either case, $M$ is a Schubert matroid by definition. Since the lattice of cyclic flats of non-Schubert matroids is not a chain, we obtain that  $r(\hat{\bm{1}}) = 2$ as required.


Property (2) follows directly from Theorem \ref{caxioms}\textbf{(Z2)}. \qedhere

\end{proof}

Using Proposition \ref{prop non schubert lof rk 2}, we observe that the number of isomorphism classes of matroids on $[n]$ of rank 2 then only depends on $|L|$, $C$ and the number of ways to partition $C$ into the $C_i$ sets.

\begin{theorem}\label{thm count rk 2}
Let $n \ge4$, and let $\rho(n)$ denote the number of unordered sums up to $n$ with components greater than or equal to 2. 
Then
\begin{align*}
|\mathscr{M}_{2, n}| = 2|\mathscr{M}_{2, n-1}| - |\mathscr{M}_{2, n-2}| + \rho(n).
\end{align*}
\end{theorem}
\begin{proof}
    First, assume that $|L(M)|>0$. Up to isomorphism, $n\in L(M)$. Then, the deletion $M\setminus n$ is a matroid over $[n-1]$, and can identify $M$ with the class $[M\setminus n]$ in $\mathscr{M}_{2, n-1}$.
    
    Now, assume that $|L(M)| = 0$ and that $|P(M)|< n$. Then $M$ has a coloop, which, up to isomorphism, is $n$. Again, $M\setminus n$ is a matroid over $[n-1]$ with the same parallel class structure. This way, we can identify $M$ with the class $[M\setminus n]$ in $\mathscr{M}_{2, n-1}$.
    Repeating the previous argument, there are $|\mathscr{M}_{2, n-1}| - |\mathscr{M}_{2, n-2}|$ loopless matroids in $\mathscr{M}_{2, n-1}$.

    Finally, assume $|L(M)|= 0$ and $|P(M)| =n.$ In this case, each isomorphism class of matroids corresponds to an unordered sum up to $n$ with components larger than 2.  This follows as two matroids $M_1$ and $ M_2$ of rank $2$ over $[n]$ are isomorphic if and only if they have the same number of parallel classes of each size. 
    
    Thus, each distinct class in $\mathscr{M}_{2, n}$ corresponds to a  distinct partition of $[n]$ into parallel classes, i.e., they each correspond to a distinct sum to $n$. There are hence $\rho(n)$ such classes. 

    Combining all cases, we obtain 
    \begin{align*}
|\mathscr{M}_{2, n}| = |\mathscr{M}_{2, n-1}| + (|\mathscr{M}_{2, n-1}| - |\mathscr{M}_{2, n-2}|) + \rho(n).
\end{align*}\qedhere
\end{proof}
\subsection{Schubert expansions of rank 2 matroids} Now that we have found all matroids of rank 2 on $[n]$ and their corresponding lattices of cyclic flats, we decompose the indicator functions of  their base polytopes into the indicator functions of base polytopes of Schubert matroids.

\begin{theorem}\label{thm schubert decomp rk 2}
    Let $M$ be a matroid of rank 2 over $[n]$. Let $l \coloneqq|L(M)|$ denote the number of loops in $M$, $p$ denote the number of parallel classes in $M$, and for each natural number $k \geq 2$, let $n_{k}$ denote the number of parallel classes in $M$ of size $k$. Then 
\begin{align*}
\mathbbm{1}(\mathcal{P}(M)) = (1-p)\mathbbm{1}(\mathcal{P}(S(\{ l+1, l+2\}))) + \sum^{n-1}_{k=2}n_{k} \mathbbm{1}(\mathcal{P}(S(\{ l+1, l+k+1 \}))).
\end{align*}
\end{theorem}

\begin{proof}


We determine the cyclic chain lattice $\mathscr{C}_{\mathcal{Z}}(M)$. The minimal element of $\mathscr{C}_{\mathcal{Z}}(M)$ is the chain $\C_0 \coloneqq L(M) \prec [n]$. The atoms are all chains of the form $\mathcal{C_i} \coloneqq L(M) \prec L(M) \sqcup C_i \prec [n]$ where $C_i$ is a parallel class of $M$. Observe that the cyclic chain lattice $\mathscr{C}_{\mathcal{Z}}(M)$ is isomorphic to the lattice of cyclic flats $\mathcal{Z}(M)$ described in Lemma \ref{prop non schubert rk 2}.

 By definition, $\hat{\bm{1}}$ covers each $\C_{i}$ for $i \neq 0$. Thus,  $\mu(\C_{i}, \hat{\bm{1}}) = -1$, meaning $\lambda_{\C_{i}} = 1$. It remains to compute $\mu(\C_0, \hat{\bm{1}})$. We have 
\begin{align*}
    \mu(\C_0, \hat{\bm{1}}) & = -(\mu(\C_{0}, \C_{0}) + \sum_{i = 1}^{p}\mu(\C_{i}, \C_{0}))  = -(1+  \sum_{i = 1}^{p}(-1)) = p -1.
\end{align*}
Consequently, we have $\lambda_{\C_{0}} = 1-p$. 

Recall that by Theorem \ref{thm schubert rk 2}, the isomorphism class of a Schubert matroid depends only on the size of its parallel classes. Thus, we sum the coefficients of isomorphic chains corresponding to isomorphic parallel classes, obtaining
\begin{align*}
\mathbbm{1}(\mathcal{P}(M)) = (1-p)\mathbbm{1}(\mathcal{P}(S(\{ l+1, l+2\}))) + \sum^{n-1}_{k=2}n_{k} \mathbbm{1}(\mathcal{P}(S(\{ l+1, l+k+1 \}))).
\end{align*}
\end{proof}

\subsection{The polytope of all matroids for {$\mathscr{M}_{2, n}$}}
Using the Schubert expansions of non-Schubert matroids described in Theorem \ref{thm schubert decomp rk 2}, we can now construct the polytope of all matroids in rank $2$. In \cite[Example 3.3]{FerroniFink2025}, it was computed that the polytope of all matroids of rank $2$ on four elements, $\Omega_{2,4}$ is given as the convex hull of
\begin{align} \label{eq omega24}
\O_{2, 4} = \begin{pmatrix}
1 & 0 & 0 & 0 & 0 & 0 & 0 \\
0 & 1 & 0 & 0 & 0 & 0 & 0 \\
0 & 0 & 1 & 0 & 0 & 0 & 0 \\
0 & 0 & 0 & 1 & 0 & 0 & 0 \\
0 & 0 & 0 & 0 & 1 & 0 & 2 \\
0 & 0 & 0 & 0 & 0 & 1 & -1 
\end{pmatrix}.
\end{align} 
Now, the matrix generating the polytope of all matroids  of rank 2 can be recursively constructed as follows. For $n \geq 4$, let $\mathcal{U}(n)$ be the matrix of $n-2$ rows with non-negative integer entries that consists of all distinct column vectors $\bm{u}_{i}$ such that
\begin{align*}
(n-1, n-2, \dots, 3, 2) \bm{u}_{i} = n.
\end{align*}
In other words, every column of $\mathcal{U}(n)$ represents an unordered sum to $n$ of natural numbers greater than or equal to 2  (excluding the trivial representation $n=n$), with repeated entries permitted. Now, define 

\begin{equation*} \mathcal{V}(n) \  = \ \ 
    \begin{tabular}{|c c|}
    \hline
    
       $ \bm{0}^{\top}  $  & \multicolumn{1}{|c|}{ }  \\ \cline{1-1} 
        & \multicolumn{1}{|c|}{$\mathcal{U}(n)$}   \\ 
        $\mathcal{V}(n-1)$ & \multicolumn{1}{|c|}{}   \\\cline{2-2}   & \multicolumn{1}{|c|}{$\bm{1}^{\top} - \bm{1}^{\top}\mathcal{U}(n)$}   \\ \hline
    \end{tabular}
\end{equation*}
\text{where we initialise with }  $\mathcal{V}(4) = \begin{pmatrix} 
0 & 2 & -1
\end{pmatrix}^T$. We then set

\begin{equation}\label{eq o2n}
    \O_{2,n} \ = \ \ \begin{tabular}{|c c c |}
    \hline
    
      $\O_{2, n-1}$& \multicolumn{2}{|c|}{ $\mathbf{0}$}  \\ \cline{1-3}
      \multicolumn{1}{|c|}{$\mathbf{0}$} &$I_{n-1}$ & \multicolumn{1}{|c|}{$\mathcal{V}(n)$} \\ \hline
      
    \end{tabular}.
\end{equation}
\begin{theorem}\label{thm schubert polytope}
The polytope $\Omega_{2,n} \coloneqq \conv(\O_{2, n})$, defined as the convex hull of the columns of the matrix in \eqref{eq o2n} is the polytope of all matroids of rank $2$ on $[n]$.
\end{theorem}
\begin{proof}
We proceed by induction, initialising with $\Omega_{2,4}$ \eqref{eq omega24}. Assume that $\Omega_{2, n}$ is the polytope of all matroids of rank $2$ on $[n]$. Then, by construction, the $|\mathscr{M}_{2, n-1}|$  leftmost columns of  $\O_{2,n}$ correspond to the $|\mathscr{M}_{2, n-1}|$  matroids over $[n]$ with loops, which exist by Theorem \ref{thm count rk 2}. 

The next $n-1$ columns correspond to the $n-1$ loopless Schubert matroids over $[n]$ ordered by the number of steps between the first and second kink of their lattices, as identified in Theorem \ref{thm schubert rk 2}(2) and (4). 

It remains to be verified that the Schubert expansions of loopless non-Schubert matroids of rank $2$ are described by the columns of $\mathcal{V}(n)$. By Theorem \ref{thm schubert decomp rk 2}, substituting $l = 0$, every loopless matroid can be expressed by the number of elements in their parallel classes. There are two cases: Either, $M$ has at least one coloop, or it is coloop-free. 

If $M$ has at least one coloop, the union of all parallel classes is strictly contained in $[n]$. Thus, $M$ can be identified with a matroid  with the same parallel class structure in $\mathscr{M}_{2, n-1}$. By induction,  matroids of this type are recorded by $\mathcal{V}(n-1)$, corresponding to the left block of columns in $\mathcal{V}(n)$. There are 
$|\mathscr{M}_{2, n-1}| - |\mathscr{M}_{2, n-2}| - (n-2)$ columns in $\mathcal{V}(n-1)$, corresponding to the columns of $\O_{2, n-1}$ representing non-Schubert matroids on $[n-1]$ with no loops.

Finally, if $M$ is both loop- and coloop-free, every element of the ground set is contained in some parallel class. By the last part of Theorem \ref{thm count rk 2}, $M$ corresponds to an unordered sum up to $n$, and its Schubert expansion can be described in terms of this sum using Theorem \ref{thm schubert decomp rk 2}. Note that this decomposition corresponds precisely to a column in the right block of $\mathcal{V}(n)$, and that there are $\rho(n)-1$ such columns. 

Now, we have 
\[
|\mathscr{M}_{2, n-1}| +(n-1) + (|\mathscr{M}_{2, n-1}| - |\mathscr{M}_{2, n-2}| + 2 - n) + (\rho(n)-1) = 2 |\mathscr{M}_{2, n-1}|  - |\mathscr{M}_{2, n-2}| + \rho(n)
\]
nonidentical columns. By Theorem \ref{thm schubert decomp rk 2}, each column encodes a Schubert expansion of an isomorphism class of matroids. Further, by Theorem \ref{thm count rk 2}, $2 |\mathscr{M}_{2, n-1}|  - |\mathscr{M}_{2, n-2}| + \rho(n)$ is the number of isomorphism classes of matroids of rank $2$ over $[n]$. Thus, the statement follows. \qedhere

\subsection{Extremal matroids of rank 2}

When considering isomorphism classes of matroids it is of particular interest to discuss their extremal ones, as positivity conjectures can be checked on these classes already. 

\begin{definition}
    A matroid $M$ of rank $r$ on $[n]$ is called \emph{extremal} if the associated point of its isomorphism class, $\Omega([M])$ is a vertex of $\Omega_{r,n}$.
\end{definition}

In \cite{FerroniFink2025}, the authors construct two nonisomorphic matroids of rank $3$ whose isomorphism classes yield the same vertex of the polytope of all matroids $\Omega_{3,6}.$ In other words, in general there are more extremal matroids than there are vertices of $\Omega_{r,n}$. However, this is not the case in rank 2:

\begin{corollary}\label{cor extremal}
No two nonisomorphic matroids of rank two have the same Schubert expansion. In particular, the vertices of $\Omega_{2,n}$ are in one-to-one correspondence with the isomorphism classes of extremal matroids.
\end{corollary}
\begin{proof}
    This follows from the proof of Theorem \ref{thm schubert polytope} as the construction of the matrix $\O_{2,n}$ implies that there are no two identical columns, and each column in $\O_{2,n}$ is in one-to-one correspondence with an isomorphism class of matroids. Hence, no two nonisomorphic matroids have the same Schubert expansion. Thus, no two nonisomorphic matroids correspond to the same lattice point in $\RR^{|\mathscr{S}_{r,n}|}$. In particular, every class of extremal matroids is in one-to-one correspondence with a vertex of $\Omega_{2,n}$
\end{proof}

\subsection{Computing the polytope of all matroids of rank 2}\label{sec implementation}

    At \begin{center}
        \url{https://github.com/VictoriaSchleis/polytope_of_all_matroids},
    \end{center} we provide code that computes the polytope of all matroids of rank $2$ in \texttt{Oscar} \cite{Oscar} for large ground set sizes. The polytope itself can be computed for ground set size $33$ on a small machine using the function \texttt{polytope\_of\_all\_matroids\_rk\_2}. It is a polytope of dimension 527 generated as the convex hull of 53929 lattice points. In Table \ref{tab: extremal matroids}, we record the numbers of vertices of all $\Omega_{2,n}$ for $n\leq 25$. 
    
    However, forgoing the computation of the convex hull and using sparse matrices, the computation can be pushed much further: On a standard computer, the matrix $\O_{2,n}$ generating $\Omega_{2,n}$ can be computed up within a few minutes up to ground set size 65 and within a few days up to ground set size of 80 using the function \texttt{polytope\_of\_all\_matroids\_rk\_2\_sparse}.

    \begin{table}[]
        \centering
        \begin{tabular}{c||c|c|c|c|c|c|c|c|c|c|c}
             \textbf{n}& 9 & 10 & 11 & 12 & 13 & 14 & 15 & 16 & 17 & 18 & 19\\ \hline \hline 
             dimension & 35 & 44 & 54 & 65 & 77 & 90 & 104 & 119 & 135  & 152 & 170 \\ \hline  \# isomorphism classes & 87 & 128 & 183 & 259& 359 & 493 & 668 & 898 & 1194 & 1578 & 2067 \\ \hline
             \# extremal matroids & 54 & 72 & 100 & 124 & 164 & 204 & 260 & 315 & 398 & 472 & 588  
        \end{tabular}
        
        \vspace{0.5cm}
        
        \begin{tabular}{c||c|c|c|c|c|c|c|c}
             \textbf{n}& 20& 21 & 22 & 23 & 24 & 25 & 26 & 27   \\ \hline \hline  
            dimension & 189 & 209 & 230 & 252 & 275 &299 & 324 & 350  \\ \hline
             \# isomorphism classes & 2693 & 3484 & 4485 & 5739 & 7313 & 9270 & 11705 & 14714 \\ \hline 
            \# extremal matroids & 686 & 831 & 970 & 1180 & 1348 
        \end{tabular}

        \vspace{0.5cm}

        \begin{tabular}{c||c|c | c| c| c}
             \textbf{n}& 28 & 29 & 30 & 31 & 32 \\ \hline \hline  
             dimension &  377 & 405 & 434 & 464 & 495 \\ \hline 
             \# isomorphism classes & 18431 & 22995 & 28598 & 35439 & 43787 
        \end{tabular}
        \
        \caption{We record the output of our computations. We omit the cases $n \leq 8$, where we recover the computations recorded in \cite[Table 1]{FerroniFink2025}. By Corollary \ref{cor extremal}, in rank $2$ the number of vertices is equal to the number of extremal matroids. }
        \label{tab: extremal matroids}
    \end{table}

\end{proof}

\section{Schubert coefficients in rank 3}\label{sec poam rk 3}
We now characterise the lattices of cyclic flats and Schubert expansions of all matroids of rank 3. Unlike matroids of rank 2, the cyclic chain lattices of matroids of rank 3 are not as easily characterisable. To simplify our analysis, we introduce different types of flats and tackle the computation in a case-by-case analysis. Every matroid in this section will be assumed to be of rank 3, over the ground set $[n]$.

\begin{definition}
    Let $M$ be a matroid and let $F\in \Z(M)$ be a cyclic flat. We say that $F$ is \emph{separable} if it can be written as the union of parallel classes. Otherwise, we say that $F$ is \emph{inseparable}. 
\end{definition}
This allows us to decompose the lattice of cyclic flats. In the following, whenever we refer to separable or inseparable flats, we are referring to flats of rank 2. 

In the remainder of this section, we compute the Schubert coefficients for all matroids of rank 3. We start by computing these coefficients for simple matroids in Section \ref{sec rk 3 only 3 cycles}, and then for all matroids by analysing how Schubert coefficients change under simplification in Section \ref{sec rk 3 overlapping}. This then allows us to construct the polytope of all matroids in rank 3 up to ground set size 12 in Section \ref{subsec poam rk 3}.

\subsection{Simple matroids of rank 3} \label{sec rk 3 only 3 cycles}
We begin our analysis by considering matroids that have no parallel classes, i.e., coloopless simple matroids. Note that in such a matroid, every rank 2 cyclic flat $F\in \Z_2(M)$ is inseparable, and further satisfies that  $M | F \cong U_{2, |F|}$.

\begin{remark}\label{rem algo insep}
The lattice of cyclic flats of a coloopless simple matroid $M$ consists precisely of $\emptyset$, all the inseparable flats, and $E(M)$. Thus, it suffices to construct all possible arrangements of inseparable flats. These arrangements satisfy two properties: 
\begin{enumerate}
    \item Every inseparable cyclic flat $F$ is of size at least 3 and at most $n-1$; and 
    \item For two inseparable cyclic flats $F_1$ and $F_2$, we have that  $|F_1 \cap F_2| \leq 1$, since the intersection of two flats is always a flat, and every rank 1 flat of size greater than or equal to $2$ is cyclic. 
\end{enumerate}
    
\end{remark}

\begin{lemma}
    Every arrangement of inseparable flats satisfying properties (1) and (2) from above is a lattice of cyclic flats of a matroid of rank $3$. 
\end{lemma}
\begin{proof}
    Follows immediately by verifying the axioms in Theorem \ref{thm properties z(m)}.
\end{proof}

We use this characterisation to computationally incrementally generate all such arrangements. For simple matroids, we can easily compute their Schubert coefficients with 

\begin{proposition}\label{prop schubert simple}
Let $M$ be a rank 3 matroid on $[n+m]$ with no parallel classes, $m$ coloops, which are on $[n+1;n+m]$, and $t$  inseparable cyclic flats, $t_{i}$ of which are of cardinality $i$. Then
\begin{enumerate}
\item $\lambda_{S} = 1-t$ for $S=S(1, 2, 3)$,
\item $\lambda_{S} = t_{k}$ for $S = S(1, 2, 1+t_{k})$,
\end{enumerate}
and $\lambda_{S} = 0$ for every other Schubert matroid.
\end{proposition}
\begin{proof}
Since only inseparable flats can be intermediate elements of $\mathcal{Z}(M)$, $\mathcal{Z}(M)$ is structured similarly to rank 2 matroids -- it consists of 3-element chains through the inseparable flats. In the chain lattice $\mathscr{C}_{\mathcal{Z}}(M)$, each 3-element chain $C$ is co-atomic and represents isomorphic matroids up to the size $k$ of the size $l$ inseparable flat contained within it. Hence $\mu(C, \hat{\bm{1}}) = -1$, and there are $t_{k}$ such chains, so $\lambda_{S} = t_{k}$. To calculate $\mu(C_{0}, \hat{\bm{1}})$, corresponding to to $S(1, 2, 3)$, observe that there are $t$ intermediate elements. For each intermediate element, observe that if $\hat{\bm{1}}$ directly covers $F$, then $\mu(F, \hat{\bm{1}}) = -1$. Thus, $\mu(C, \hat{\bm{1}}) = t-1$ and so $\lambda_{S(1,2,3)}=1-t$. These are the only two chain types in the chain lattice, so all other Schubert coefficients are zero. \qedhere
\end{proof}

Note that this proposition still holds if there are no inseparable flats. In this case, the sole non-zero Schubert coefficient is $\lambda_{(1,2,3)} = 1.$

\subsection{The general case}\label{sec rk 3 overlapping}

To calculate the Schubert coefficients of a general matroid, we now describe how the Schubert coefficients change under simplification. For this, it will be helpful to distinguish cyclic flats of rank 1 by whether they are contained in an inseparable flat or not. For the remainder of this section we will assume a matroid $M$ to be on $[n+m]$, where $M$ has $m$ coloops, which are in $[n+1;n+m]$.

\begin{definition}\label{def matroid type k}
Let $M$ be a matroid of rank 3. We say that a rank 1 cyclic flat is \emph{contained} if it is contained in an inseparable cyclic flat. 
\end{definition}

\subsubsection{Simplification and insertion}
It is a classical result that every matroid has a unique \emph{simplification}, i.e., a matroid with the same lattice of flats, but where all loops and all but one element of each parallel class is deleted. We characterised such matroids and their Schubert coefficients in the previous Section \ref{sec rk 3 only 3 cycles}. Now, we consider how the Schubert coefficients of a matroid can be determined from its simplification. To this end, we discuss the process of undoing simplification, called parallel insertion. 

\begin{definition}  
Let $M$ be a matroid on $[n+m]$ and $e\in [n+m]$. We define the \textit{parallel insertion of $M$ at $e$}, denoted $M_{e}$, to be the matroid on $[n+m+1]$ whose flats are those of $M$, except any flat $F \ni e$ now also satisfies $n+m+1 \in F$. 
\end{definition}

\begin{example}\label{ex: insertion 1}
    Consider the loopless matroid $ M$ of rank $3$ on $[6]$ depicted on the left below:
    \begin{center}
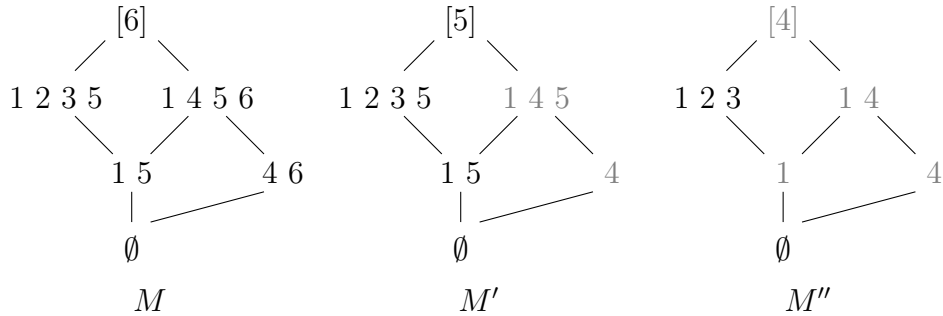
    \begin{tabular}{ccc}
       
\Lazypic{4cm}{
\begin{tikzpicture}[scale = 1]
\draw (0.25, -0.75) -- (0.75, -0.25);
\draw (-0.25, -0.75) -- (-0.75, -0.25);
\draw (0.75, 0.35) -- (0.35, 0.75);
\draw (-0.75, 0.35) -- (-0.35, 0.75);
\draw (1.25, -0.25) -- (1.75, -0.75);
\draw (0, -1.25) -- (0,-1.65);
\draw (1.75, -1.25) --(0.25, -1.65);
\draw (0, -2) node{$\emptyset$};
\draw (0, -1) node{$1 \ 5$};
\draw (2, -1) node{$4 \ 6$};
\draw (-1, 0) node{$1 \ 2\ 3 \ 5$};
\draw (1, 0) node{$1\ 4\ 5\ 6$};
\draw(0, 1) node{$[6]$};
\end{tikzpicture}}  &  
\Lazypic{4cm}{
\begin{tikzpicture}[scale = 1]
\draw (0.25, -0.75) -- (0.75, -0.25);
\draw (-0.25, -0.75) -- (-0.75, -0.25);
\draw (0.75, 0.35) -- (0.35, 0.75);
\draw (-0.75, 0.35) -- (-0.35, 0.75);
\draw (1.25, -0.25) -- (1.75, -0.75);
\draw (0, -1.25) -- (0,-1.65);
\draw (1.75, -1.25) --(0.25, -1.65);
\draw (0, -2) node{$\emptyset$};
\draw (0, -1) node{$1 \ 5$};
\draw (2, -1) node{\color{gray} $4$};
\draw (-1, 0) node{$1\ 2\ 3 \ 5$};
\draw (1, 0) node{\color{gray}$1\ 4\ 5$};
\draw(0, 1) node{$[5]$};
\end{tikzpicture}} & 
\Lazypic{4cm}{
\begin{tikzpicture}[scale =1]
\draw (0.25, -0.75) -- (0.75, -0.25);
\draw (-0.25, -0.75) -- (-0.75, -0.25);
\draw (0.75, 0.35) -- (0.35, 0.75);
\draw (-0.75, 0.35) -- (-0.35, 0.75);
\draw (1.25, -0.25) -- (1.75, -0.75);
\draw (0, -1.25) -- (0,-1.65);
\draw (1.75, -1.25) --(0.25, -1.65);
\draw (0, -2) node{$\emptyset$};
\draw (0, -1) node{\color{gray}$1$};
\draw (2, -1) node{\color{gray}$4$};
\draw (-1, 0) node{$1 \ 2 \ 3$};
\draw (1, 0) node{\color{gray}$1\ 4$};
\draw(0, 1) node{\color{gray}$[4]$};
\end{tikzpicture}}\\
$M$ & $M'$ & $M''$
    \end{tabular}
    \end{center}
We observe that this matroid arises from the middle matroid $M'$ by inserting $6$ at $4$, obtaining the parallel class $16$, and considering the resulting lattice structure. Note that this introduces \emph{two} new cyclic flats, and changes a third - all flats of $M'$ involving $1$ turn into cyclic flats of $M$ by taking the union with $6$, whether they were cyclic before or not. 

Similarly, the matroid $M'$ arises from the matroid $M''$ on the right by inserting $5$ into $1$. This creates a parallel class $45$, and turns the ground set from a non-cyclic flat to a cyclic flat. Note that in the above lattices, we depict non-cyclic flats that turn into cyclic flats in gray.
\end{example}

\begin{remark}
Let $M$ be a loopless rank 3 matroid on $[n+m]$, where $k$ is the number of coloops of $M$, and let $e \in [n+m]$. By isomorphism, we may assume that the coloops are the elements $[n+1;n+m]$. If $e$ is a coloop,  the lattice $\mathcal{Z}(M_{e})$ is obtained from $\mathcal{Z}(M)$ by adding a new uncontained rank 1 cyclic flat $\{e, n+m+1\}$, and for each $P \in \mathcal{Z}_{1}(M)$, a new separable cyclic flat $P \cup \{e, n+m+1\}$. Additionally, if $e$ is the only coloop of $M$, the flat $[n+m+1]$ is a cyclic flat. We compute the Schubert coefficients of this parallel insertion in case (1) in Section \ref{sec case 1}. 

If $e$ is not a coloop, there exists a  cyclic flat $F \in \mathcal{Z}(M)$ of minimal rank containing $e$, and one of the following is true:
\begin{enumerate}
\item[(2)] $F$ is a (not necessarily unique) inseparable cyclic flat (see Section \ref{sec case 2});
\item[(3)] $F$ is a unique contained cyclic flat (see Section \ref{sec case 3}); or
\item[(4)] $F$ is a unique uncontained cyclic flat (see Section \ref{sec case 4}).
\end{enumerate}
\label{ecases}
This is because $E(M)$ is a cyclic flat by the position of the coloops. If $e$ is contained in a separable cyclic flat $G$, then it is also contained in a rank 1 cyclic flat. Further, if $F$ is of rank 1, it is unique as there is a unique flat of rank 1 containing a ground set element since $M$ is loopless.

Here, the insertion changes the lattice of cyclic flats as follows:
\begin{itemize}
\item[(2)] In case 2, where $\mathscr{F}$ is the set of inseparable cyclic flats containing $e$, $\mathcal{Z}(M_{e})$ is obtained by introducing a new contained rank 1 cyclic flat $\{e, n+m+1\}$, replacing every flat $F\in \mathscr{F}$ with $F \cup \{n+1\}$, and introducing new separable cyclic flats of the form $P \cup \{e, n+1\}$ for each $P \in \mathcal{Z}_{1}(M)$, as in case (1).
\item[(3)+(4)] In cases 3 and 4, let $F\ni e$ be the unique flat containing $e$. Then $\mathcal{Z}(M_{e})$ is obtained by exchanging all flats $\Tilde{F} \supseteq F$ for $\Tilde{F} \cup \{n+m+1\}$.
\end{itemize}

\end{remark}

\begin{example}\label{ex insertion 2}
    In Example \ref{ex: insertion 1}, the insertion of $5$ into $1$ creates another Schubert matroid and is a case 1 insertion. However, since the original matroid does not have a maximal chain of length $4$ in the lattice of cyclic flats, we will cover this type of insertion in Section \ref{sec exceptional cases}. The insertion of $6$ into $4$ is also a case 1 insertion, covered in Section \ref{sec case 1} as it is not exceptional. 
\end{example}

\begin{lemma}
    Any matroid $M$ of rank $3$ on $[n+m]$ can be constructed from its simplification via parallel insertion. 
\end{lemma}
\begin{proof}
    This follows immediately as parallel insertion is the reversal of a simplification step, and every matroid has a unique simplification. Note that due to our notation convention, it will be advantageous to perform parallel insertions into coloops at the end of the process. 
\end{proof}

\subsubsection{Schubert coefficients}
We now compute the Schubert coefficients of arbitrary matroids using the extension discussed in the previous subsection. Note that this is a finite process as simplification is finite, and the simplification has lower ground set size. We begin with a preliminary lemma that will aid in the computation of the minimal chain. 

\begin{lemma} \label{lem lattice changes}
Let $\mathcal{L}$ be a ranked, finite lattice whose elements take ranks $\{0, 1, 2, 3\}$. Let $v_{0}$ be the minimal element and $v_{3}$ the maximal element, and $V_{1}, V_{2}$ the sets of elements of ranks 1 and 2 respectively, and let $C$ be the set of 3 element saturated chains from $v_{0}$ to an element of $V_{2}$. Then
\begin{align*}
\mu(v_{0}, v_{3}) = |V_{1}|+|V_{2}|-|C|-1.
\end{align*}
\end{lemma}
\begin{proof}
Let $C_{v}$ be the set of elements of $C$ whose maximal element is $v$, so $C = \bigcup_{r(v)=2}C_{v}$. Then
\begin{align*}
\mu(v_{0},v_{3}) &= - \sum_{v_{0} \leq v < v_{3}} \mu(v_{0}, v) \ \ =\ \ - \Big( \mu(v_{0}, v_{0}) + \sum_{r(v)=1} \mu(v_{0}, v) + \sum_{r(v)=2} \mu(v_{0}, v) \Big) \\
&= - \Big(1 + \sum_{r(v)=1} -1 + \sum_{r(v)=2} |C_{v}| - 1\Big) \ \  = \ \ -1 + |V_{1}| - |C| + |V_{2}|. \qedhere
\end{align*}
\end{proof}
The consequence of Lemma \ref{lem lattice changes} is that every chain from the minimal element of $\chainZ(M)$ increases the Schubert expansion coefficient of $S(1,2,3)$, and every non minimal, non maximal element decreases it. In the computations that follow, denote by $\zeta_{C}$ the Schubert expansion coefficient of the chain $C$ (or analogously, the Schubert matroid S) in $M_{e}$ to distinguish from the decomposition coefficients of $M$.
\begin{definition}
Let $S=S(i_{1}, i_{2}, \dots, i_{r}) \in \mathscr{S}_{n, r}$. Then define 
\begin{align*}
S' = S(i_{1}, \dots, i_{r}) \in \mathscr{S}_{n, r}.
\end{align*}
\end{definition}
$S'$ is the Schubert matroid such that $\mathcal{Z}(S')$ is identical to $\mathcal{Z}(S)$, except from the inclusion of $n+1$ in the maximal element of $\mathcal{Z}(S')$. It is immediate that if $S_{1}, S_{2} \in \mathscr{S}_{n, r}$ are not isomorphic, then $S_{1}'$ and $S_{2}'$ are not isomorphic. Moreover, we may denote $S'$ for $S=(i_1, \dots, i_{n-1}, n+1)$ to denote the corresponding Schubert matroid of $\mathscr{S}_{n+1, r}$, even though the initial expression does not describe a Schubert matroid in $\mathscr{S}_{n, r}$

We now begin with the case-by-case analysis. For simplicity of notation, we will assume that the lattice of cyclic flats of the matroid $M$ into which we insert has a chain of maximal length, i.e., of length $4$. We cover the other cases separately in Section \ref{sec exceptional cases}.

\subsubsection{Case 1}\label{sec case 1}
We begin by considering parallel insertions which produce a new, uncontained parallel class. For notational clarity, we denote the Schubert coefficients of the insertion by $\zeta_S$ for respective Schubert matroids $S$.
\begin{proposition}\label{prop schubert case 1}
Let $M$ be a loopless rank 3 matroid and let the insertion of $e \in E(M)$ into $M$ be a case 1 insertion. Denote by $p$ the number of parallel classes and by $p_k$ the number of parallel classes of size $k$. We then have
\begin{enumerate}
\item $\zeta_{S'} = \lambda_{S} + 1 - |\mathcal{P}(M)| - p_{2}$ for $S=S(1,3,4)$.
\item $\zeta_{S'} = \lambda_{S} - p_{k}$ for $S=(1, k+1,k+2)$ and $3 \leq k \leq n-1$.
\item $\zeta_{S'} = \lambda_{S} +1 - |\mathcal{P}(M)| $ for $S=S(1, 2, 3)$.
\item $\zeta_{S'} = \lambda_{S} + p_{k}$ where $S=S(1, 3, k+3)$ and $3\leq k \leq n-1$.
\item $\zeta_{S'} = \lambda_{S} + p_{k}$ where $S=S(1, k+1, k+3)$ and $3<k \leq n-1$.
\item $\zeta_{S'} = \lambda_{S} + 2p_{2}$ where $S=S(1, 3, 5)$.
\item $\zeta_{S'} = \lambda_{S} - p_{k}$, where $S=S(1, 2, k+3)$ and $2 \leq k<n-2$.
\end{enumerate}
and $\zeta_{S'} = \lambda_{S}$ in all other cases. Note that all $p_k$ are taken of the pre-insertion matroid $M$.\label{case1prop}
\end{proposition}
\begin{proof}
$M_{e}$ adds the new uncontained rank 1 cyclic flat $P_{e} = \{e, n+1\}$. Hence $\chainZ(M_{e})$ contains the new chains: 
\begin{enumerate}
\item $\emptyset < P_{e} < E(M_{e})$ of type $S(1, 3, 4)$ and Schubert coefficient $1-|\mathcal{P}(M)|$; there are $|\mathcal{P}(M)|$ chains from this chain element to $\hat{\bm{1}}$. This chain adds 1 to the coefficient of $S(1,2,3)$.
\item $\emptyset < P_{e} < P \cup P_{e} < E(M_{e})$ of type $S(1, 3, k+3)$ and a Schubert coefficient of 1 for any $P \in \mathcal{P}(M)$ with $|P|=k$; there are $p_{k}$ of these chains. It covers the two atomic chains $\emptyset < P_{e} < E(M_e)$ and $\emptyset < P \cup P_{e} < E(M_e)$, so each of the $|P(M)|$ chains of this type decrease the coefficient of $S(1,2,3)$ by 1.
\item $\emptyset < P < P \cup P_{e} < E(M_{e})$ of type $S(1, k+1, k+3)$ and Schubert coefficient $1$ for any $P \in \mathcal{P}(M)$ with $|P|=k$. Similarly to above, each of these $|\mathcal{P}(M)|$ chains decrease the Schubert coefficient of $S(1,2,3)$ by 1.
\item $\emptyset < P \cup P_{e} < E(M_{e})$ of type $S(1,2,k+3)$ for a parallel class $P \in \mathcal{P}(M)$ with $|P|=k$; there are $p_{k}$ of these, and since they are covered by two chains; $\emptyset < P < P \cup P_{e} < E(M_{e})$ and $\emptyset < P_{e} < P \cup P_{e} < E(M_{e})$, their Schubert coefficient is $-1$. Each of these $|\mathcal{P}(M)|$ chains add 1 to the coefficient of $S(1,2,3)$.
\end{enumerate}
Moreover, for $|P|=k$, any chain of the form $\emptyset < P < E(M_e)$ is now covered by the chain $\emptyset<P < P \cup P_{e} < E(M_e)$, decreasing the Schubert coefficient of $S(1,k+1,k+2)$ by $1$ per such flat. No other chains are altered, so these are the only changes to $\chainZ(M)$ and hence the Schubert coefficients.
\end{proof}

\begin{example}
    We compute the Schubert coefficients of the matroid $M$ over $[6]$ given in Example \ref{ex: insertion 1}. In Example \ref{ex insertion 2} we checked that both insertions that generate $M$ are case 1 insertions. Since $M'$ is isomorphic to the Schubert matroid $S(1,3,4)$, we can read off the sole Schubert coefficient: $\lambda_{S(1,3,4)} = 1$. Similarly, read off that $|P(M')| = p_2 =1$, and that all other $p_k=0$. 
    
    From this, we can use Proposition \ref{prop schubert case 1} to determine the coefficients of $M$: \begin{itemize}
        \item[(1)] We have $\zeta_{S} = 1+1-1-1 = 0$ for $S = S(1,3,4)$;
        \item[(3)] We have $\zeta_{S}= 0 +1 -1 = 0 $ for $S = S(1,2,3)$; 
        \item[(6)] We have $\zeta_{S} = 0+2 = 2$ for $S = S(1,3,5)$; and
        \item[(7)] We have $\zeta_S = 0-1 = -1$ for $S = S(1,2,5)$.
    \end{itemize}
    Note that all other possible computations yield $0$, since $\lambda_S = 0$ for all $S \neq S(1,3,4)$ and since $p_k = 0$ for all $k>2$. 
\end{example} 
\subsubsection{Case 2}\label{sec case 2} Next, we consider the case where parallel insertion produces a new, contained parallel class. Since the cases (2)-(4) do not change the coloops of $M$, for simplicity of notation we will assume that $M$ has no coloops, and consider matroids with such coloops when constructing the polytope of all matroids in Section \ref{subsec poam rk 3}. Recall that we denote the Schubert coefficients of the insertion by $\zeta_S$ for respective Schubert matroids $S$. 

Cases 2 and 3 rely on additional variables. If $M$ is a rank 3 matroid, let $p^{c}_{k}$ and $p^{u}_{k}$ be the number of contained and uncontained rank 1 cyclic flats respectively, so that $p_{k} = p^{c}_{k} + p^{u}_{k}$. 

\begin{proposition}
Let $M$ be a loopless matroid of rank 3 with no coloops and $e \in E(M)$ such that $e$ is contained in no rank 1 cyclic flats, and $f$ inseparable rank 2 cyclic flats. Let $p^{e}_{k}$ be the number of rank 1 cyclic flats of $M$ that aren't contained in some inseparable rank 2 cyclic flat that $e$ is contained in. Then
\begin{enumerate}
\item $\zeta_{S'} = \lambda_{S} + 1 - f - p^{e}$ for $S=S(1,2,3)$.
\item $\zeta_{S'} = \lambda_{S} + 1 - f - p^{e} - p^{e}_{2}$ for $S=S(1, 3, 4)$.
\item $\zeta_{S'} = \lambda_{S} - p^{e}_{k}$ for $S=S(1,k+1,k+2)$ and $3 \leq k \leq n-1$.
\end{enumerate}
Moreover, set
\begin{enumerate}
\item $\zeta_{S'} = \lambda_{S} - p^{e}_{k}$ for $S=S(1, 2, k+3)$.
\item $\zeta_{S'} = \lambda_{S} + p^{e}_{k}$ for $S=S(1, 3, k+3)$ and $k > 2$.
\item $\zeta_{S'} = \lambda_{S} + p^{e}_{k}$ for $S=S(1,k+1, k+3)$ and $k>2$.
\item $\zeta_{S'} = \lambda_{S} + 2p^{e}_{2}$ for $S=S(1,3,5)$,
\end{enumerate}
and $\zeta_{S'} = \lambda_{S}$ for all other Schubert matroids. Now, run through each inseparable rank 2 cyclic flat of $M$ that contains $e$. Say this flat contains $p^{f}$ contained rank 1 cyclic flats, $p^{f}_{k}$ of size $k$. Then
\begin{enumerate}
\item $\zeta_{S'} \mapsto \zeta_{S'} + 1$ for $S=S(1, 2, |F|+2)$.
\item $\zeta_{S'} \mapsto \zeta_{S'} + p^{f}-1$ for $S=S(1, 2, |F|+1)$.
\item $\zeta_{S'} \mapsto \zeta_{S'} - p^{f}$ for $S=S(1, 2, |F|+2)$.
\item $\zeta_{S'} \mapsto \zeta_{S'} - p^{f}_{k}$ for $S=S(1, k+1, |F|+1)$.
\item $\zeta_{S'} \mapsto \zeta_{S'} + p^{f}_k$ for $S(1, k+1, |F|+2)$.
\end{enumerate}

\end{proposition}
\begin{proof}
We add the new chains:
\begin{enumerate}
\item $\emptyset < \{e, n+1\} < E(M_{e})$ of the form $S(1,3,4)$. This is covered by $p^{e}+f$ other chains (where we add a rank 2 cyclic flat into the chain which is either $P \cup \{e, n+1\}$ for a rank 1 cyclic flat which isn't contained in an inseparable rank 2 cyclic flat which contains $e$, or $F \cup \{n+1\}$ for some inseparable rank 2 cyclic flat containing $e$), so its Schubert coefficient is $1-f-p^{e}$, and it contributes 1 to the Schubert coefficient of $S(1,2,3)$.
\item $\emptyset <\{e, n+1\} < F \cup \{n+1\} < E(M_{e})$ of the form $S(1, 3, |F|+2)$ for an inseparable rank 2 cyclic flat $F$ containing $e$, which has Schubert coefficient 1. There are $f$ of these, and they are ordered above the chains $\emptyset < F \cup \{n+1\} < E(M_{e})$. They cover two chains (obtained by removing one of the intermediate chain elements) so they contribute $-1$ to the Schubert coefficient of $S(1,2,3)$, and there are $f$ of them.
\item $\emptyset < P \cup \{e, n+1\} < E(M_{0})$ of the form $S(1, 2, k+3)$ for any rank 1 cyclic flat $P$ of $M$ with $|P|=k$, which isn't contained in a rank 2 inseparable cyclic flat containing $e$; there are $p^{e}$ of these total, and $p^{e}_{k}$ of this specific form. They are covered by two chains (each obtained by adding $P$ or $\{e, n+1\}$ into the chain) so their Schubert coefficient is $-1$, and they add $1$ to the Schubert coefficient of $S(1, 2, 3)$.
\item $\emptyset < \{e, n+1 \} < P \cup \{e, n+1\} < E(M_{e})$ of the form $S(1, 3, k+3)$ for any rank 1 cyclic flat $P$ that isn't contained in some inseparable rank 2 cyclic flat that $e$ is contained in; there are $p^{e}$ of these, $p^{e}_{k}$ of each type, eith Schubert coefficient $1$. They contain two chains, so they contribute $-1$ to the Schubert coefficient of $S(1, 2, 3)$.
\item $\emptyset < P < P \cup \{e, n+1\} < E(M_{e})$ for $P$ as above. There are $p^{e}$ of these, $p^{e}_{k}$ of the type $S(1, k+1, k+3)$. Similar to above, their Schubert coefficient is $1$ and they contribute $-1$ to the Schubert coefficient of $S(1, 2, 3)$. It is also important to note that this chain covers the chain $\emptyset < P < E(M_{e})$, hence decreasing its Schubert coefficient by 1 (and hence decreasing the Schubert coefficient of $S(1,k+1,k+2)$ by $p_{k}$).
\end{enumerate}
We then need to run through the flats containing $e$; since we add $\{n+1\}$ to these flats, the chains are changed in the following way (assuming notation in the statement of the proposition):
\begin{enumerate}
\item $\big( \emptyset < F < E(M) \big) \to \big( \emptyset < F \cup \{n+1\} < E(M_{e}) \big)$, mapping a chain of type $S(1, 2, |F|+1)$ to one of the type $S(1, 2, |F|+2)$; this chain initially had Schubert coefficient $1-p^{f}$ since it was covered by $p^{f}$ chains (obtained by adding in a contained rank 1 cyclic flat below $F$ which is contained in $F$); we need to add $p^{f}-1$ to $\zeta_{S'}$ for $S=S(1,2, |F|+1)$. The addition of the rank 1 cyclic flat $\{e, n+1\}$ decreases the Schubert coefficient of this chain by 1 relative to its original coefficient as it is covered by a new chain. Hence we subtract $p^{f}$ from the Schubert coefficient $\zeta_{S'}$ for $S=S(1, 2, |F|+2)$.
\item $\big( \emptyset < P < F < E(M) \big) \to \big( \emptyset < P < F \cup \{n+1\} < E(M_{e}) \big)$, mapping a chain of type $S(1,k+1, |F|+1)$ to one of type $S(1, k+1, |F|+2)$ for some rank 1 cyclic flat $P \subset F$ with $|P|=k$. Its Schubert coefficient is 1 and there are $p^{f}_{k}$ of these so we subtract $p^{f}_{k}$ from $\zeta_{S'}$ for $S=S(1, k+1, |F|+1)$ and add $p^{f}_{k}$ to $\zeta_{S'}$ for $S=S(1, k+1, |F|+2)$.
\end{enumerate}
\end{proof}

\subsubsection{Case 3} \label{sec case 3} Next, we consider the case where parallel insertion does not alter the shape of the lattice, but instead inserts an element into an existing contained parallel class. Recall that since the cases (2)-(4) do not change the coloops of $M$, we assume that $M$ has no coloops for ease of notation, and that we denote the Schubert coefficients of the insertion by $\zeta_S$ for respective Schubert matroids $S$. 

Recall that we write $p^{c}_{k}$ and $p^{u}_{k}$ for the number of contained and uncontained rank 1 cyclic flats respectively, so that $p_{k} = p^{c}_{k} + p^{u}_{k}$. 

\begin{proposition}
Let $M$ be a loopless matroid of rank 3 with no coloops and $e \in P_{e}$ for some contained rank 1 cyclic flat $P_{e}$ of size $m$, which is contained in $f$ inseparable cyclic flats. We have
\begin{enumerate}
\item $\zeta_{S'} = \lambda_{S} + p^{u} + f - 1$ for $S=S(1,m+1, m+2)$; and
\item $\zeta_{S'} = \lambda_{S} + 1 - p^{u} - f$ for $S=S(1, m+2, m+3)$.
\end{enumerate}
 Write $\C\F_e = \{F \in \C\F(M) | e\in F\}$, and set $p_k^e = |\{P \nsubseteq F \suchthat P\in \Z_1(M) \text{ and } F\in \C\F_e\}|$. That is, $p^{e}_{k}$ is the number of parallel classes that are not contained in an inseparable cyclic flat that $e$ is contained in. We then have
\begin{enumerate}
\item $\zeta_{S'} = \lambda_{S} + p^{e}_{k-1} - p^{e}_{k}$ for $S=S(1, k+1, m+k+1)$ for $k \neq m, m+1$.
\item $\zeta_{S'} = \lambda_{S} - p^{e}_{k}$ for $S=S(1, m+1, m+k+1)$ for $k \neq m, 0$.
\item $\zeta_{S'} = \lambda_{S} + p^{e}_{k}$ for $S=S(1, m+2, m+k+1)$ for $k \neq m, 0$.
\item $\zeta_{S'} = \lambda_{S} - p^{e}_{m+1} + 2p^{e}_{m} - 2$ for $S=S(1, m+2, 2m+2)$
\item $\zeta_{S'} = \lambda_{S} - 2p^e_{m} + p^e_{m-1} + 2$ for $S=S(1, m+1, 2m+1)$.
\item $\zeta_{S'} = \lambda_{S} +p^{e}_{k} - p^{e}_{k-1}$ for $S=S(1, 2, m+k+1)$.
\item $\zeta_{S'} = \lambda_{S} + p^{e}_{m} - 1 - p^{e}_{m-1}$ for $S=S(1, 2, 2m+1)$.
\item $\zeta_{S'} = \lambda_{S} + p^{e}_{m+1} - p^{e}_{m}+1$ for $S=S(1, 2, 2m+2)$.
\end{enumerate}
and $\zeta_{S'} = \lambda_{S}$ for all other Schubert matroids. Run through every element of $\{F \in i\mathcal{F}(M) \suchthat P_{e} \subset F \}$ in the following way; for each $F$, say $F$ contains $p^{f}$ rank 1 cyclic flats (including $P_{e}$); $p^{f}_{k}$ of size $k$. Then
\begin{enumerate}
\item $\zeta_{S'} \mapsto \zeta_{S'} + p^{f} - 1$ for $S=S(1, 2, |F|+1)$.
\item $\zeta_{S'} \mapsto \zeta_{S'} + 1 - p^{f}$ for $S=S(1, 2, |F|+2)$.
\item $\zeta_{S} \mapsto \zeta_{S'} -1$ for $S=S(1, m+1, |F|+1)$.
\item $\zeta_{S} \mapsto \zeta_{S'} + 1$ for $S=S(1, m+2, |F|+2)$.\
\item $\zeta_{S} \mapsto \zeta_{S'} -p^{f}_{k}$ for $S=S(1, k+1, |F|+1)$.
\item $\zeta_{S} \mapsto \zeta_{S'} + p^{f}_{k}$ for $S=S(1, k+1, |F|+2)$, $k \neq m$.
\item $\zeta_{S} \mapsto \zeta_{S'} + p^{f}_{m} - 1$ for $S=S(1, m+1, |F|+2)$.
\end{enumerate}
The result after running through all such flats is the Schubert expansion of $M_{e}$.

\end{proposition}
\begin{proof}
The first change is
\begin{align*}
\big( \emptyset < P_{e} < E(M)\big) \to \big( \emptyset < P_{e} \cup \{n+1\} < E(M_{e})).
\end{align*}
This changes a chain of type $S(1, m+1, m+2)$ to one of the type $S(1, m+2, m+3)$; the original chain (hence both) have Schubert coefficient $1-f-p^{e}$ since it is covered by $p^{e}+f$ chains, obtained by either adding $P \cup P_{e}$ or $F$ above $P_{e}$ in the chain, where $P$ is a parallel class which isn't contained in an inseparable cyclic flat $P_{e}$ is contained in, and $F$ is an inseparable cyclic flat containing $P_{e}$. Hence we add $1-f-p^{e}$ to $\zeta_{S'}$ for $S=S(1, m+2, m+3)$ and subtract it from $\zeta_{S'}$ for $S=S(1, m+1, m+2)$. Since no new elements or chains are added in $\mathcal{Z}(M)$, the Schubert coefficient for $S(1,2,3)$ is unchanged. Over the parallel classes $P \subset M$ which are not contained in a rank 1 cyclic flat $P_{e}$ is contained in, we have the changes similar to case 1. Now over each inseparable cyclic flat $F$ containing $P_{e}$ with notation as in the statement of the proposition, we have the changes
\begin{enumerate}
\item $\big( \emptyset < F < E(M) \big) \to \big( \emptyset < F \cup \{n+1\} < E(M_{e}) \big)$; mapping a matroid of type $S(1, 2, |F|+1)$ to one of type $S(1, 2, |F|+2)$. It has Schubert coefficient $1-p^{f}$ so we subtract this from $\zeta_{S'}$ for $S=S(1, 2, |F|+1)$ and add it to $\zeta_{S'}$ for $S=S(1, 2, |F|+2)$.
\item $\big( \emptyset < P_{e} < F < E(M) \big) \to \big( \emptyset < P_{e} \cup \{n+1\} < F \cup \{ n + 1 \} < E(M_{e})$; mapping a matroid of type $S(1, m+1, |F|+1)$ to one of type $S(1, m+2, |F|+2)$. This has Schubert coefficient $1$, so subtract 1 from $\zeta_{S'}$ for $S=S(1, m+1, |F|+1)$ and add 1 to $\zeta_{S'}$ for $S=S(1, m+2, |F|+2)$.
\item $\big( \emptyset < P < F < E(M) \big) \to \big( \emptyset < P < F \cup \{n+1\} \big)$; mapping a matroid of type $S(1, k+1, |F|+1)$ to one of type $S(1, k+1, |F|+2)$ with $P \subset F$ a parallel class other than $P_{e}$ contained in $F$ with $|P|=k$. It has Schubert coefficient 1, and $p^{f}_{k}$ exist for $k \neq m$, and $p^{f}_{m}$ exist for $k=m$, so if $k \neq m$ we subtract $p^{f}_{k}$ from $\zeta_{S}$ for $S=S(1, k+1, |F|+1)$ and add $p^{f}_{k}$ to $S=S(1, k+1, |F|+2)$; if $k=m$ we instead subtract $p^{f}_{m}-1$ to the former and add $p^{f}_{m} -1$ to the latter.
\end{enumerate}
\end{proof}
\subsubsection{Case 4} \label{sec case 4} Finally, we consider the case where parallel insertion does not alter the shape of the lattice, but instead inserts an element into an exisiting uncontained parallel class. Recall that we assume that the matroid we insert into has no loops or coloops, and that we denote the Schubert coefficients of the insertion by $\zeta_S$ for respective Schubert matroids $S$.
\begin{proposition}
Let $M$ be a matroid and  let $e \in P_{e}$ for some uncontained rank 1 cyclic flat $P_{e}$ of size $m$. Then
\begin{enumerate}
\item $\zeta_{S'} = \lambda_{S} -p_{k} + p_{k-1}$ for $S=S(1, k+1, m+k+1)$ and $k \neq m, m+1$.
\item $\zeta_{S'} = \lambda_{S} - p_{k}$ for $S=S(1, m+1, m+k+1)$ and $k \neq m$
\item $\zeta_{S'} = \lambda_{S} + p_{k}$ for $S=S(1, m+2, m+k+2)$ and $k \neq m$
\item $\zeta_{S'} = \lambda_{S} -p_{m+1} + 2p_{m} - 2$ for $S=S(1, m+2, 2m+2)$.
\item $\zeta_{S'} = \lambda_{S} - 2p_{m} + p_{m-1} + 2$ for $S=S(1, m+1, 2m+1)$.
\item $\zeta_{S'} = \lambda_{S} + p - 2$ for $S=S(1, m+1, m+2)$.
\item $\zeta_{S'} = \lambda_{S} + 2 - p$ for $S=S(1, m+2, m+3)$.
\item $\zeta_{S'} = \lambda_{S} + p_{k} - p_{k-1}$ for $S=S(1, 2, m+k+1)$ and $m \neq k, k-1$.
\item $\zeta_{S'} = \lambda_{S} + p_{m} - 1 - p_{m-1}$ for $S=S(1,2,2m+1)$.
\item $\zeta_{S'} = \lambda_{S} + p_{m+1} - p_{m}+1$ for $S=S(1,2,2m+2)$.
\end{enumerate}

\end{proposition}
\begin{proof}
As in the proof of Proposition \ref{case1prop}, we do not add any new elements or chains into $\mathcal{Z}(M)$, so we do not add any new elements or chains into $\chainZ(M)$, hence $S'(1,2,3) = S(1,2,3)$. The changes being made with the parallel insertion are, that for any $P \in \mathcal{Z}_{1}(M)$ with $|P|=k$ we have, 
\begin{enumerate}
\item $\big(\emptyset < P_{e} < P_{e} \cup P < E(M)\big) \to \big( \emptyset < P_{e} \cup \{n+1\} < P_{e} \cup P \cup \{n+1\} < E(M_{e})\big)$. This maps a chain of type $S(1, m+1, m+k+1)$ to one of type $S(1, m+2, m+k+2)$, whose Schubert coefficient was 1. There are $p_{k}$ of these chains if $k \neq m$; $p_{m}-1$ if $p=m$.
\item $\big( \emptyset < P < P \cup P_{e} < E(M)\big) \to \big(\emptyset < P < P \cup P_{e} \cup \{n+1\} < E(M_{e})\big)$. This maps a chain of type $S(1, k+1, m+k+1)$ to one of type $S(1, k+1, m+k+2)$. The counts of these correspond to those above, and their Schubert coefficients are also 1.
\item $\big( \emptyset < P \cup P_{e} < E(M) \big) \to \big(\emptyset <P \cup P_{e} \cup \{n+1\} < E(M_{e})\big)$. This maps a chain of type $S(1,2,m+k+1)$ to one of type $S(1, 2, m+k+2)$. The counts of these are the same as above, and their Schubert coefficient is $-1$ since they are covered by two chains.
\end{enumerate}
\end{proof}

\subsubsection{Exceptional cases.}\label{sec exceptional cases} Finally, we need to consider the case where there exist no chains of length $4$ in the lattice of cyclic flats of the matroid $M$ into which we insert. This is the case where there are no contained parallel classes, and less than 2 parallel classes. In particular, these are:
\begin{itemize}
\item A case $1$ or $2$ insertion on a matroid with no parallel classes.
\item A case $2$ or $3$ insertion on a matroid with one parallel class, which is uncontained. 
\item A case $4$ insertion on a matroid with one parallel class and no indecomposable cyclic flats.
\end{itemize}
These resolve in the following way:
\begin{enumerate}
\item A case 1 insertion on a matroid with no parallel classes increases the Schubert coefficient of $S(1,3,4)$ by $1$ and decreases the Schubert coefficient of $S(1,2,3)$ by 1.
\item A case 2 insertion on a matroid with no parallel classes, or one uncontained parallel class, into a $3$-cycle class of size $k$ increases the Schubert coefficient of $S(1,2,k+1)$ by $1$ and decreases the Schubert coefficient of $S(1,2,k)$ by $1$.
\item A case 3 insertion on a matroid with one uncontained parallel class of size $k$ increases the Schubert coefficient of $S(1,k+2, k+3)$ by $1$ and decreases the Schubert coefficient of $S(1,k+1, k+2)$ by $1$.
\item A case 4 insertion on a matroid with one parallel classes and no indecomposable flats increases the Schubert coefficient of $S(1,4,5)$ by 1 and decreases the Schubert coefficient of $S(1,3,4)$ by 1.
\end{enumerate}


\section{The polytope of all matroids for {$\mathscr{M}_{3, n}$}}\label{subsec poam rk 3}
In this section we construct the polytope of all matroids in rank 3. Each subsection transforms the Schubert computations done in the previous section into a matrix component. 

As in rank 2, we construct a matrix $\O_{3,n}$, such that $\Omega_{3,n} \coloneqq \conv (\O_{3,n})$ is the polytope of all matroids. 

\subsection{Loops and coloops}
In the preceeding sections, we described how to compute the polytope of all loopless, coloopless matroids of rank 3 on $E(M)$. We now cover the case where a matroid $M$ has loops, and the case where $M$ has coloops. This proceeds fairly similarly to the rank 2 analysis of the same setting. As such, we will significantly shorten the exposition here and refer to \ref{sec poam rk 2} for details.

\[\O^\ell_{3,n} \coloneqq \ \ \begin{array}{|c|}\hline
\O_{3,n-1} \\ \hline
\mathbf{0}\\\hline
\end{array}.
\]
Next, assume $M$ has at least one coloop $c$. Since coloops are not contained in any cyclic flat, the structure of the lattice of cyclic flats is preserved under deletion of $c$. Then, the deletion $M \setminus c$ is isomorphic to a matroid of rank $2$ on $n-1$, and conversely, every matroid of rank $3$ with at least one coloop can be constructed in this way. Thus, the matrix corresponding to all matroids of rank 3 on $[n]$ with coloops is 
\[\O^{c\ell}_{3,n} \coloneqq \ \ \begin{array}{|c|}\hline
\mathbf{0}\\\hline
\O_{2,n-1} \\ \hline
\end{array}.
\]

\subsection{Computing the polytope of all matroids of rank 3}

To set up the final computation, we now need to compute $\O_{3,4}$ by hand as our initial input for the recursive generation. We additionally compute $\O_{3,5}$ by hand to demonstrate our algorithm at work.

\begin{example}[$\O_{3,3}$ and $\mathcal{O}_{3,4}$]
    We observe that $\Omega_{3,3}$ is a point: The smallest rank $3$ matroid is $U_{3,3}$, which is a Schubert matroid and the only element of $\mathscr{M}_{3,3}$.

    Similarly, every element of  $\M_{3,4}$ is a Schubert matroid.   There are $3$ such Schubert matroids:
    \begin{itemize}
        \item First, we have that $\mathscr{M}_{2,3} = \mathscr{S}_{2,3}$. Every addition of a coloop to a representative of an isomorphism class in $\M_{2,3}$ now yields a new, non-isomorphic matroid isomorphism class in $\M_{3,3}$:
        \begin{itemize}
            \item $S(1,2)\oplus U_{1,1}$ is isomorphic to $S(1,2,4)$;
            \item $S(1,3)\oplus U_{1,1}$ is isomorphic to $S(1,3,4)$;
            \item $S(2,3)\oplus U_{1,1}$ is isomorphic to $S(2,3,4)$;
        \end{itemize}
        \item Next, we have the Schubert matroid $S(1,2,3)$; and 
        \item Finally, we have the matroid $U_{3,3}\oplus U_{0,1}$, i.e., the matroid $U_{3,3}$ with an added loop. This matroid is isomorphic to $S(2,3,4)$. 
    \end{itemize}
    Any insertion on a uniform matroid must be an insertion on $E(M)$, which provides a Schubert matroid, leaving us with no further ways to form an element of $\mathscr{M}_{3,4}$. 
\end{example}

We conclude by giving an overview over our final algorithm and our implementation.

\begin{algorithm}
\caption{\texttt{polytope\_of\_all\_matroids\_rk\_3}}\label{alg:cap}
\begin{algorithmic}
\Require A number $n\in \NN_{\geq 4}$.
\Ensure The matrix $\O_{3,n}$.
\State Initialise a matrix \texttt{M}
\For{$i \in \{n, ...,1\}$}
\State $L =$ \texttt{cycle\_covers(i)}
\For{$c \in L$}
\State \texttt{append!(M, unique!(schubert\_coeffs\_insep(c,n,i)))}
\EndFor
\EndFor
\State \texttt{append!(M, $\begin{bmatrix}
    \texttt{polytope\_of\_all\_matroids\_rk\_3(n-1)} \\0 
\end{bmatrix}$)}
\State \texttt{append!(M, $\begin{bmatrix}
    0 \\
    \texttt{polytope\_of\_all\_matroids\_rk\_2(n-1)}
\end{bmatrix}$)}
\State Return \texttt{M}
\end{algorithmic}
\end{algorithm}

In Algorithm \ref{alg:cap}, \texttt{cycle\_covers} computes all possible arrangements of inseparable flats as described in Remark \ref{rem algo insep}. The function \texttt{schubert\_coeffs\_insertion(c,n,i))} then computes all new Schubert coefficients arising in the \texttt{i} possible insertions into \texttt{c}, using the case-by-case computations discussed in Sections \ref{sec case 1} - \ref{sec exceptional cases}.

Finally, we combine all results in this section, and obtain our main theorem.

\begin{theorem}\label{thm main rk 3}
    Let $n \in \NN_{\geq 4}$, and let $\O_{3,n}$ be the matrix arising from Algorithm \ref{alg:cap}. The polytope $\Omega_{3,n} \coloneqq \conv ( \O_{3,n})$ is the polytope of all matroids of rank $3$ on $[n]$.
\end{theorem}

We demonstrate our algorithm by computing $\O_{3,5}$ manually. Over $[5]$, we see the first non-Schubert matroids of rank 3:

\begin{example}[$\mathcal{O}_{3,5}$]
There are $10$ Schubert matroids of rank $3$ on $[5]$, along with a trivial matroid with a coloop which is not a Schubert matroid, obtained by adding a coloop to the sole matroid of $\mathscr{M}_{2,4} \setminus \mathscr{S}_{2,4}$, discussed in Example \ref{thm schubert rk 2}. 

Next, we consider insertion on $\mathscr{M}_{3,4}$. Since insertion disregards loops and $S(2,3,4)$ is isomorphic to $U_{3,3}$ with an added loop, insertion on $S(1,2,3) = U_{3,4}$ or $S(2,3,4)$ produces a Schubert matroid.

There are two remaining Schubert matroids in $\mathcal{S}_{3,4}$. We now consider insertions on these. We obtain two Schubert and two non-Schubert matroids this way, and depict the lattices of cyclic flats of the latter in Figure \ref{fig 35 insertions}

\begin{itemize}
\item[(a)] The lattice of cyclic flats of $S(1,3,4)$  contains a singular parallel class of size 2. Here, insertion can either increase the size of the parallel class (causing the matroid to remain a Schubert matroid) or add a new parallel class, which creates a non-Schubert matroid. 
\item[(b)] The lattice of cyclic flats of $S(1,2,4)$ contains a single inseparable flat of size $3$. Here, insertion either adds a new parallel class contained in the inseparable flat, or adds a new uncontained parallel class. The former is a Schubert matroid, but the latter is not. 
\end{itemize}
\begin{figure}
\begin{tikzpicture}
\coordinate (A) at (0,-1);
\coordinate (B) at (-1,0);
\coordinate (C) at (1,0);
\coordinate (D) at (0,1);
\coordinate (E) at (0,2);
\draw (A) -- (B) -- (D) -- (E)
(A) -- (C) -- (D)
(A) node[fill=white]{$\emptyset$}
(B) node[fill=white]{$1 \ 2$}
(C) node[fill=white]{$3 \ 4$}
(D) node[fill=white]{$1 \ 2 \ 3 \ 4$}
(E) node[fill=white]{$[5]$};
\end{tikzpicture}
\hspace{2cm}
\begin{tikzpicture}
\coordinate (A) at (0,-1);
\coordinate (B) at (1,0);
\coordinate (C) at (-1,1);
\coordinate (D) at (0,2);
\draw (A) -- (B) -- (D)
(A) -- (C) -- (D)
(A) node[fill=white]{$\emptyset$}
(B) node[fill=white]{$1 \ 2$}
(C) node[fill=white]{$3 \ 4 \ 5$}
(D) node[fill=white]{$[5]$};
\end{tikzpicture}
\caption{Lattices of cyclic flats for the non-Schubert matroids obtained via insertion on $S(1,3,4)$ (left) and $S(1,2,4)$ (right).} \label{fig 35 insertions}
\end{figure}
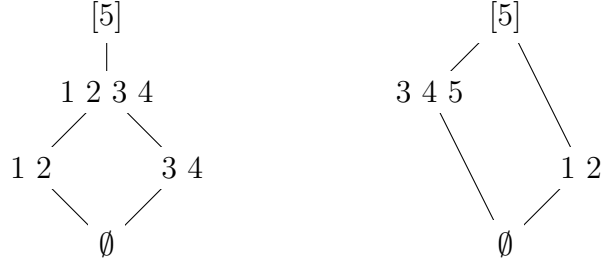

We can now calculate the Schubert coefficients for the two non-Schubert matroids constructed above; the matroid arising in (a) is a case 1 insertion and the one in (b) is an exceptional case 2 insertion.
\begin{itemize}
\item[(a)] In case (a) we determine that $p_{2}=p=1$ and $p_{k}=0$ for all other $k$. Hence 
\begin{align*}
    \zeta_{S(1,2,3)'} = 0, \ \ \ \ \ \ \ \  \zeta_{S(1,3,4)'} = \lambda_{S(1,3,4)'} + 1- p-p_{2} = 0, \\ \zeta_{S(1,3,5)'} = 2p_{2}=2 \ \ \ \ \ \ \text{and} \ \ \ \ \ \ \zeta_{S(1,2,5)'} = -p_{2}=-1 
\end{align*}

\item[(b)] In case (b) we increase $\lambda_{S(1,3,4)}$ by 1 and decrease $\lambda_{S(1,2,3)}$ by 1, leaving us with the nonzero Schubert coefficients
\[
\lambda_{S(1,2,3)}=-1, \ \ \ \  \lambda_{S(1,3,4)}=1 \ \ \ \ \text{and} \ \ \ \ \lambda_{S(1,2,4)} = 1.
\]
\end{itemize}
Finally, the case $n=5$ is the first one to provide  nontrivial inseparable flats; $\{1, 2, 3\}$ and $\{ 3,4,5 \}$. From this lattice structure, in the notation of Proposition \ref{prop schubert simple} we have that $t =2$ and $t_k = (0,0,2,0...)$. Thus,  
\[\lambda_{S(1,2,3)} =1 - t = -1 \  \ \ \ \ \ \text{and} \ \ \ \ \ \ \lambda_{S(1,2,4)} = \lambda_{S(1,2,1+3)} = t_3 =  2 \]

In total, these three matroids, plus the consideration of the earlier matroid obtained by adding a coloop, gives the matrix
\begin{equation*}
\mathcal{O}_{3,5} = 
 \color{gray}
\begin{matrix}
S(3,4,5) \\
S(2,4,5) \\
S(2,3,5) \\
S(2,3,4) \\
S(1,4,5) \\
S(1,3,5) \\
S(1,3,4) \\
S(1,2,5) \\
S(1,2,4) \\
S(1,2,3) 
\end{matrix}
\color{black}
\left(
I_{10} \  \ \
\begin{matrix}
 0& 0 & 0 & 0\\
 0 & 0 & 0 & 0\\
 0 & 0 & 0 & 0\\
 0 & 0 & 0 & 0\\
 0  & 0 & 0 & 0\\
 0  &0 & 2 & 0\\
 0  & 1 & 0 & 0\\
 0  & 0 & -1 & 0\\
 2  & 1 & 0 & 2\\
 -1 &-1 & 0 & -1\\
\end{matrix}
\right)
\end{equation*}
where we notice that the matroid obtained by adding a coloop shares decomposition coefficients from the matroid with 2 inseparable flats.
\end{example}



\paragraph{Computations}
Computationally implementing Algorithm \ref{alg:cap}, our process to compute $\O_{3,n}$ terminates within a few minutes for $n\leq 10$, and with some additional patience for $n = 11$ using the code provided at \begin{center}
        \url{https://github.com/VictoriaSchleis/polytope_of_all_matroids}.
    \end{center} Due to the blow-up in the number of isomorphism classes of matroids, we expect the cases $n\geq 12$ to be computationally out of reach at this point in time. This is a memory problem, rather than an algorithmic one - the computation of all irreducible flat arrangements is already out of reach for $n=12$. This is not unexpected, rather it was surprising to us that the polytope can be computed as far as it can. We refer to the jupyter notebook available at the code source for a detailed explanation and documentation of the implementation. 

Similarly, the polytope of all matroids can be computed using the matrix $\O_{3,n}$ and the \texttt{convex\_hull}-functionality in \texttt{OSCAR} \cite{Oscar} for all $n\leq 10$. Computing all vertices of this polytope is both memory- and computationally heavy, but can be done with some patience for $\Omega_{3,n\leq 8}$. 
    
As we have done in rank 2, we record some of the resulting invariants in Table \ref{tab: rk 3}. 

\begin{table}
        \centering
        \begin{tabular}{c||c|c|c|c|c|c}
             \textbf{n}& 5 & 6 & 7 & 8 & 9 & 10 \\ \hline \hline 
             dimension & 10 & 20 & 35 & 56 & 84 & 120   \\ \hline  \# distinct Schubert expansions & 13 & 39 & 109 & 310 & 960  & 3291 \\ \hline
             \# vertices & 11 & 28 & 64 & 145 &  &    \\
        \end{tabular}
    \caption{The number of distinct Schubert expansions of matroids in $\M_{3,n}$ compared to the number of vertices of $\Omega_{3,n}$ for $3\leq n \leq 10$. Note that for $n\leq 8$, we recover the results of \cite[Table 1]{FerroniFink2025}.}
    \label{tab: rk 3}
\end{table}

 \bibliographystyle{amsalpha}
\bibliography{bibliography.bib}

	\vspace{0.25cm}
	\noindent
    \textsc{Narayan Collins, Durham University, Upper Mountjoy Campus,
		DH1 3LE Durham,
		United Kingdom} \\
	\emph{Email :} narayan.j.collins@durham.ac.uk
	\vspace{0.25cm}
	
	\noindent
	\textsc{Victoria Schleis, Durham University, Upper Mountjoy Campus,
		DH1 3LE Durham,
		United Kingdom} \\
	\emph{Email :} victoria.m.schleis@durham.ac.uk

\end{document}